\documentclass[a4paper,10pt
,reqno]{amsart}

\usepackage{amssymb,amsmath,amsthm,ascmac,array,bm,multirow,appendix}
\usepackage[dvipdfmx]{graphicx}
\usepackage{xcolor}
\usepackage{lscape}

\usepackage[
colorlinks=true,bookmarks=true,
bookmarksnumbered=true,bookmarkstype=toc,linktocpage=true
]{hyperref}

\allowdisplaybreaks
\numberwithin{equation}{section}

\usepackage{mathtools}
\mathtoolsset{showonlyrefs=true}

\newtheorem{thm}{Theorem}[section]

\newtheorem{lem}[thm]{Lemma}

\newtheorem{rem}[thm]{Remark}
\newtheorem{ass}[thm]{Assumption}

\newcommand{\mcc}{\mathcal{C}}

\newcommand{\mcf}{\mathcal{F}}
\newcommand{\mci}{\mathcal{I}}

\newcommand{\mcl}{\mathcal{L}}

\newcommand{\mbbh}{\mathbb{H}}

\newcommand{\mbbr}{\mathbb{R}}

\newcommand{\mbbu}{\mathbb{U}}
\newcommand{\mbby}{\mathbb{Y}}
\newcommand{\mbbz}{\mathbb{Z}}

\newcommand{\sig}{\sigma}
\newcommand{\ep}{\epsilon}
\newcommand{\D}{\Delta}
\newcommand{\Sig}{\Sigma}
\newcommand{\lam}{\lambda}

\newcommand{\Gam}{\Gamma}
\newcommand{\p}{\partial}

\newcommand{\cil}{\xrightarrow{\mcl}} 
\newcommand{\cip}{\xrightarrow{p}} 

\newcommand{\argmin}{\mathop{\rm argmin}} 
\newcommand{\argmax}{\mathop{\rm argmax}}
 
\newcommand{\trace}{\mathop{\rm trace}} 

\def\ds#1{\displaystyle{#1}} 

\def\nn{\nonumber}
\def\cadlag{c\`adl\`ag}

\def\sumj{\sum_{j=1}^{n}}

\def\pr{P}
\def\E{E}

\def\dim{\mathrm{dim}}

\def\tz{\theta_{0}}

\def\tes{\hat{\theta}_{n}}

\newcommand{\nY}{\overline{\mathsf{y}}} 

\newcommand{\betabic}{\mathrm{dpGQBIC}}
\newcommand{\gambic}{\mathrm{HGQBIC}}

\title[Robustified Gaussian quasi-BIC for volatility]
{Robustified Gaussian quasi-BIC for volatility}

\author[S. Eguchi]{Shoichi Eguchi}
\address{Faculty of Information Science and Technology, Osaka Institute of Technology, 1-79-1 Kitayama, Hirakata City, Osaka, 573-0196, Japan.}
\email{shoichi.eguchi@oit.ac.jp}

\author[H. Masuda]{Hiroki Masuda}
\address{Graduate School of Mathematical Sciences, University of Tokyo, 3-8-1 Komaba Meguro-ku Tokyo 153-8914, Japan.}
\email{hmasuda@ms.u-tokyo.ac.jp}

\date{\today}

\begin{document}
\setlength{\baselineskip}{4.5mm}

\maketitle

\begin{abstract}
We develop a theoretical foundation for robust model comparison in a class of non-ergodic continuous volatility regression models contaminated by finite-activity jumps. 
Using the density-power weighting and the H\"{o}lder(-inequality)-based normalization of the conventional Gaussian quasi-likelihood function, we propose two Schwarz-type statistics 
and also establish their model selection consistency with respect to the minimal true parametric volatility coefficient. Numerical experiments are conducted to illustrate our theoretical findings.
\end{abstract}

\section{Introduction}

Suppose that we are given a complete filtered probability space $(\Omega,\mcf,(\mcf_{t})_{t\in[0,T]},P)$ for a fixed time horizon $[0,T]$, on which the $d$-dimensional {\cadlag} process
\begin{align}
Y_{t}&=Y_{0}+\int_{0}^{t}\mu_{s-}ds+\int_{0}^{t}\sigma(X_{s-})dw_{s}+J_{t},
\label{se:trueY}
\end{align}
and the $d^{\prime}$-dimensional {\cadlag} process
\begin{align*}
X_{t}&=X_{0}+\int_{0}^{t}\mu_{s-}^{\prime}ds+\int_{0}^{t}\sigma_{s-}^{\prime}dw_{s}^{\prime}+J_{t}^{\prime}
\end{align*}
are defined.
The components are specified as follows:
\begin{itemize}
\item $\sigma:\mbbr^{d^{\prime}}\to\mbbr^{d}\otimes\mbbr^{r}$;
\item $\mu$, $\mu^{\prime}$, and $\sigma^{\prime}$ are processes in $\mbbr^{d}$, $\mbbr^{d^{\prime}}$, and $\mbbr^{d^{\prime}}\otimes\mbbr^{r^{\prime}}$, respectively;
\item $w$ and $w^{\prime}$ are standard Wiener processes in $\mbbr^{r}$ and $\mbbr^{r^{\prime}}$, respectively;
\item $J$ and $J^{\prime}$ are adapted finite-activity pure-jump in $\mbbr^d$ and $\mbbr^{d'}$, respectively.
\end{itemize}
As in the observations of $Y$ and $X$, we consider a discrete but high-frequency sample $\{(X_{t_{j}},Y_{t_{j}})\}_{j=0}^{n}$, where $t_{j}=t_{j}^{n}:=jh$ with $h=h_{n}:=T/n$.

In this paper, we are interested in relative model comparison of the parametric diffusion coefficient $\sigma$ in the model \eqref{se:trueY}.
Suppose that the candidate diffusion coefficients are given by
\begin{align*}
\sigma_{1}(x,\theta_{1}),\ldots,\sigma_{M}(x,\theta_{M}),
\end{align*}
where, for each $m\in\{1,\ldots,M\}$, $\theta_{m}\in\Theta_{m}\subset\mbbr^{p_{m}}$, and the parameter space $\Theta_{m}$ is assumed to be a bounded convex domain.
Then, for each $m\in\{1,\ldots,M\}$, the $m$-th candidate model $\mathcal{M}_{m}$ is described by
\begin{align}
Y_{t}&=Y_{0}+\int_{0}^{t}\mu_{s}ds+\int_{0}^{t}\sigma_{m}(X_{s-},\theta_{m})dw_{s}+J_{t}.
\label{se:cand.model}
\end{align}
The objective of this paper is to develop a Bayesian information criterion (BIC, \cite{Sch78}) for selecting the best model among $\mathcal{M}_{1}, \dots, \mathcal{M}_{M}$, treating observations affected by the jump-process components $J$ and $J'$ as \textit{dynamic outliers}.
That is, we want to select the best diffusion coefficient among the candidates, ignoring the jump components $J$ and $J'$.
To the best of our knowledge, there are no previous studies about a theoretical foundation for a robustified BIC in the contaminated volatility regression model \eqref{se:trueY}.


For the reader's convenience, we briefly review the relevant literature.
The information criteria are one of the most convenient and powerful tools for model selection, and the Akaike information criterion (AIC, \cite{Aka73}) and the BIC are often used.
These criteria are derived from different classical principles.
The AIC selects the model that minimizes the Kullback-Leibler divergence, which measures the discrepancy between the true model and the predictive model.
On the other hand, the BIC selects a model by maximizing the posterior probability given the data.
Based on the classical principles of AIC and BIC, several studies have investigated the model selection for stochastic differential equations (SDE) and robustified model selection; for example, \cite{Uch10}, \cite{FujUch14}, \cite{UchYos16}, \cite{EguMas18a}, \cite{EguMas24}, \cite{EguUeh21} have studied the model selection problem for SDEs, and \cite{KurHam18}, \cite{KurHam20}, \cite{Kur24} have studied the robustified model selection problem.
\cite{EguMas18a} proposed a BIC-type information criterion based on a stochastic expansion of the marginal quasi-log-likelihood and applied it to continuous semimartingales such as \eqref{se:cand.model}. 
\cite{EguMas18a} also showed the asymptotic properties of the proposed criterion.
\cite{KurHam18} and \cite{KurHam20} proposed the AIC- and BIC-type information criteria based on density-power divergence and proved their asymptotic properties.

Parameter estimation in candidate models is essential for deriving the information criterion, and several studies closely related to the present work have been conducted.
For statistical inference for SDEs based on the Gaussian quasi-likelihood function, see \cite{GenJac93}, \cite{UchYos13}, and references therein.
Moreover, \cite{LeeSon13} and \cite{Son17} investigated robust statistical inference for diffusion processes using density-power divergence, and \cite{EguMas25} studied robust statistical inference under model settings similar to those considered in this paper based on density-power and H\"{o}lder-based divergences.

The remainder of this paper is organized as follows. Section \ref{sec:pre} introduces the notation and the model setup. In Section \ref{sec:bic}, we propose our BIC-type statistics, which are robust against finite-activity jump variations; we then analyze their asymptotic properties to establish model selection consistency. Section \ref{sec:simu} presents numerical experiments that corroborate our theoretical findings. The technical proofs are given in Section \ref{sec:prf}. Finally, for the reader's convenience, we list the key technical tools in Section \ref{sec:appe}.

\section{Preliminaries} \label{sec:pre}

\subsection{Basic notation and setup}
For notational convenience, we introduce the following notations.
For any matrix $A$, we write $A^{\otimes 2}=AA^\top$, where $\top$ denotes transposition.
For a $K$th-order multilinear form $M=\{M^{(i_1\dots i_K)}:i_k=1,\dots, d_k;k=1,\dots,K\}\in\mbbr^{d_1}\otimes\dots\otimes\mbbr^{d_K}$ and $d_k$-dimensional vectors $u_k=\{u_k^{(j)}\}$, we set $M[u_1,\dots,u_K]:=\sum_{i_1=1}^{d_1}\dots\sum_{i_K=1}^{d_K} M^{(i_1,\dots,i_K)}u_1^{(i_1)}\dots u_K^{(i_K)}$. 
In particular, for matrices $A$ and $B$ of the same sizes, $A[B]:=\trace(AB^{\top})$ in case of $K=2$.
The symbol $\p_{a}^{k}$ stands for $k$-times partial differentiation with respect to variable $a$, and $I_{r}$ denotes the $r\times r$-identity matrix.
Moreover, the symbols $\cip$ and $\cil$ denote the convergence in probability and distribution, respectively.

The basic model setting is as follows.
Omitting the model index ``$m$" in \eqref{se:cand.model}, we consider the single model
\begin{align}
Y_{t}&=Y_{0}+\int_{0}^{t}\mu_{s}ds+\int_{0}^{t}\sigma(X_{s-},\theta)dw_{s}+J_{t},
\label{se:SDE.model}
\end{align}
where the diffusion coefficient depends on an unknown parameter $\theta\in\Theta\subset\mbbr^{p}$, and the parameter space $\Theta$ is assumed to be a bounded convex domain.
Let $\tz\in\Theta$ denote the true value of $\theta$, and we assume that $\sigma(\cdot,\tz)=\sigma(\cdot)$.
For a process $\xi$, define $\D_{j}\xi:=\xi_{t_{j}}-\xi_{t_{j-1}}$, and for any measurable function $f:\mbbr^{d}\times\Theta$, set $f_{j-1}(\theta):=f(X_{t_{j-1}},\theta)$.
We also define $S(x,\theta):=\sigma^{\otimes2}(x,\theta)$.

We denote by $\pr_\theta$ the distribution of the random elements
\begin{equation}
\left(Y,X,\mu,\mu',\sig',w,w',J,J'\right)
\nonumber
\end{equation}
associated with $\theta\in\overline{\Theta}$, and we write $\pr=\pr_{\tz}$.
The $d$-dimensional normal $N_d(\mu,\Sig)$-density is denoted by $\phi_d(\cdot;\mu,\Sig)$ and simply $\phi(\cdot):=\phi_d(\cdot;0,I_d)$.

\subsection{Robustified Gaussian quasi-likelihood inference}
\cite{EguMas25} studied robust statistical inference in a model setting similar to ours.
In this section, following \cite{EguMas25}, we briefly review the robustified Gaussian quasi-likelihood inference for \eqref{se:SDE.model} with jump contamination.

For the robustified Gaussian quasi-likelihood inference, we consider the density-power divergence 
from the true distribution $gd\mu$ to the statistical model $f_{\theta}d\mu$ defined by 
\begin{align*}
(f_{\theta};g) \mapsto \frac{1}{1+\lambda}\int\left\{f_{\theta}^{1+\lambda}-\left(1+\frac{1}{\lambda}\right)f_{\theta}^{\lambda}g+\frac{1}{\lambda}g^{1+\lambda}\right\}d\mu
\end{align*}
for some dominating $\sigma$-finite measure $\mu$, where $\lam=\lambda_n \in(0,\overline{\lam}]$ is a positive tuning parameter satisfying 
\begin{equation}
    \lam_n \to 0,\qquad n\to\infty.
\end{equation}
The speed of $\lam_n\to 0$ cannot be so fast (Assumption \ref{hm:lam.condition-2}); it is also possible to consider a fixed $\lam>0$ (Remark \ref{hm:rem_fixed.lam}), although in this case, the meaning of the marginal quasi-likelihood loses its natural interpretation. 
Here, the upper bound $\overline{\lam}>0$ is given, while we do not need to specify it in practice.
Applying this to our setting, we deal with the density-power weighting of the Gaussian-quasi likelihood function (GQLF) of \eqref{se:SDE.model}(\cite{GenJac93}, \cite{UchYos13}), and the density-power GQLF is defined by
\begin{align}
\mbbh_{n}(\theta;\lambda)
&=\sumj \det\big(S_{j-1}(\theta)\big)^{-\lambda/2} \left\{\frac{1}{\lambda}\phi\big(S_{j-1}(\theta)^{-1/2}\nY_j\big)^\lambda-\frac{(2\pi)^{-d\lambda/2}}{(\lambda+1)^{1+d/2}} \right\},
\label{hm:def_dp-H-0}
\end{align}
where $\nY_j=h^{-1/2}\D_{j}Y$.
Moreover, we define the H\"{o}lder-based GQLF \cite{EguMas25}: 
\begin{align}
\mbbh_{n}^{\flat}(\theta;\lambda)
=\sumj\frac{1}{\lambda}\det\big(S_{j-1}(\theta)\big)^{-\lambda/(2(\lambda+1))} \phi\big(S_{j-1}(\theta)^{-1/2}\nY_j\big)^{\lambda}.
\label{hm:def_ldp-H-0}
\end{align}
This is constructed from the H\"{o}lder inequality:
given two densities $f$ and $g$ with respect to a reference measure $\mu$ and a constant $\lam>0$, we have
\begin{equation}
    \int f^\lam g d\mu \le \bigg(\int f^{\lam+1} d\mu\bigg)^{\lam/(\lam+1)} \bigg(\int g^{\lam+1} d\mu\bigg)^{1/(\lam+1)},
\end{equation}
from which we have
\begin{equation}
\left(\int g^{\lam+1} d\mu\right)^{1/(\lam+1)} 
- \int \frac{f^\lam }{(\int f^{\lam+1}d\mu)^{\lam/(\lam+1)}} \,g d\mu \ge 0,
\end{equation}
where the equality holds if and only if $g=f$ a.e., thus defining a divergence between $f$ and $g$.

Given the value $\lambda$, the density-power Gaussian quasi-likelihood estimator (GQMLE) $\hat{\theta}_{n}(\lambda)$ and H\"{o}lder-based GQMLE $\hat{\theta}_{n}^{\flat}(\lambda)$ are defined by
\begin{align*}
\hat{\theta}_{n}(\lambda)\in\argmax_{\theta\in\bar{\Theta}}\mbbh_{n}(\theta;\lambda)
\end{align*}
and
\begin{align*}
\hat{\theta}_{n}^{\flat}(\lambda)\in\argmax_{\theta\in\bar{\Theta}}\mbbh_{n}^{\flat}(\theta;\lambda),
\end{align*}
respectively.
Under Assumptions \ref{hm:A_diff.coeff}--\ref{hm:A_lam}, the asymptotic mixed normality of density-power and H\"{o}lder-based GQMLEs is established in \cite[Theorem 3.4]{EguMas25}.
Regarding numerical performance, simulation studies in \cite{EguMas25} show that both estimators achieve reasonable performance, while the H\"{o}lder-based GQMLE can be computed more rapidly than the density-power GQMLE.

\begin{rem}
Let $\eta \sim N_p(0,I_p)$ be independent of $\mathcal{F}$.
Under Assumptions \ref{hm:A_diff.coeff}--\ref{hm:A_lam}, it follows from \cite[Theorem 3.4]{EguMas25} that the density-power and H\"{o}lder-based GQMLEs are asymptotically mixed normal:
\begin{align*}
\sqrt{n}(\hat{\theta}_{n}(\lambda)-\tz)&\cil \mathcal{I}(\tz)^{-1/2}\eta \sim MN_{p}\left(0,\mathcal{I}(\tz)^{-1}\right), \\
\sqrt{n}(\hat{\theta}_{n}^{\flat}(\lambda)-\tz)&\cil \mci(\tz)^{-1/2}\eta \sim MN_{p}\left(0,\mathcal{I}(\tz)^{-1}\right),
\end{align*}
where $\mathcal{I}\mathcal(\tz)=(\mathcal{I}_{kl}(\tz))_{k,l=1}^{p}$ is defined by
\begin{align*}
\mathcal{I}_{kl}(\tz)
=\frac{1}{2T}\int_{0}^{T}\trace\big( S^{-1} (\p_{\theta_k}S) S^{-1} (\p_{\theta_l} S)(X_{t},\tz) \big) dt.
\end{align*}
\end{rem}

\section{Gaussian quasi-BIC} \label{sec:bic}

Building on the density-power GQLF and the H\"{o}lder-based GQLF, we turn to Schwarz's type model comparison.
Let $\pi$ be the prior density for $\theta$.

\begin{ass}
\label{se:prior}
The prior density $\pi$ is bounded on $\Theta$, continuous, and positive at $\tz$.
\end{ass}

\subsection{Gaussian quasi-Bayesian information criterion}
The classical BIC methodology is based on a stochastic expansion of the marginal log-likelihood function.
For the derivation of the BIC-type information criterion, we consider the free energies at inverse temperature $\mathfrak{b}>0$ (see \cite{Wat13} for relevant background), defined as
\begin{align*}
\mathfrak{F}_{n}(\mathfrak{b};\lambda)
=-\log\left[\int_{\Theta}\exp\left\{\mathfrak{b}\left(\frac{1}{h^{d\lambda/2}}\mbbh_{n}(\theta;\lambda)-\frac{n}{\lambda}+\frac{n}{h^{d\lambda/2}}\right)\right\}\pi(\theta)d\theta\right]
\end{align*}
and
\begin{align*}
\mathfrak{F}_{n}^{\flat}(\mathfrak{b};\lambda)
=-\log\left[\int_{\Theta}\exp\left\{\mathfrak{b}\left(\frac{1}{\lam}\left(\mathsf{k}_{\lambda}\mbbh^{\flat}_n(\theta;\lam)-n\right)\right)\right\}\pi(\theta)d\theta\right],
\end{align*}
where
\begin{align*}
\mathsf{k}_{\lambda}
=(h^{d\lambda/2})^{-1/(\lam+1)} \lam \left\{  \frac{(2\pi)^{-d\lambda/2}}{(\lambda+1)^{d/2}} \right\}^{-\lam/(\lam+1)}.
\end{align*}
According to \cite[Remarks 3.1 and 3.2]{EguMas25}, both $\frac{1}{h^{d\lambda/2}}\mbbh_{n}(\theta;\lambda) - \frac{n}{\lambda} + \frac{n}{h^{d\lambda/2}}$ and $\frac{1}{\lam}\left(\mathsf{k}_{\lambda}\mbbh^{\flat}_n(\theta;\lam) - n \right)$ converge almost surely to the conventional GQLF $\mbbh_{n}(\theta)$ as $\lambda\to0$ with $n$ fixed.
Here $\mbbh_{n}(\theta)$ denotes the GQLF for $Y$ without jumps, which is defined as follows (\cite{UchYos13}):
\begin{align*}
\mbbh_{n}(\theta)=\sumj \log \phi_{d}\left(Y_{t_j};\, Y_{t_{j-1}},\, h S_{j-1}(\theta)\right).    
\end{align*}
Hence, the random functions $\mathfrak{F}_{n}(1;\lambda)$ and $\mathfrak{F}_{n}^{\flat}(1;\lambda)$ can be regarded as the marginal quasi-log-likelihood functions associated with the density-power GQLF and the H\"{o}lder-based GQLF, respectively.

The following theorem gives the stochastic expansions of $\mathfrak{F}_{n}$ and $\mathfrak{F}_{n}^{\flat}$, showing that heating-up is necessary to obtain the appropriate stochastic expansions.

\begin{thm} 
\label{se:thm_denbic}
Suppose that Assumptions \ref{hm:A_diff.coeff}--\ref{hm:A_lam} and \ref{se:prior} hold.
Then, we have
\begin{align}
\mathfrak{F}_{n}(h^{d\lambda/2};\lambda)
&=-\mbbh_{n}\big(\tes(\lambda);\lambda\big)+\frac{p}{2}\log n 
+\frac{nh^{d\lambda/2}}{\lambda}-n+O_{p}(1), \label{se:st_exp_den} \\
\mathfrak{F}_{n}^{\flat}(\lambda/\mathsf{k}_{\lambda};\lambda)
&=-\mbbh_{n}^{\flat}\big(\hat{\theta}_{n}^{\flat}(\lambda);\lambda\big)+\frac{p}{2}\log n 
+\frac{n}{\mathsf{k}_{\lambda}}+O_{p}(1). \label{se:st_exp_hol}
\end{align}
\end{thm}

In view of Theorem \ref{se:thm_denbic}, by multiplying both sides by $2$, we obtain the following:
\begin{align*}
2\mathfrak{F}_{n}(h^{d\lambda/2};\lambda)
&=-2\mbbh_{n}\big(\tes(\lambda);\lambda\big)+p\log n+\frac{2nh^{d\lambda/2}}{\lambda}-2n+O_{p}(1).
\end{align*}
Ignoring the $O_{p}(1)$ term, which is asymptotically negligible relative to the other terms, we define the density-power Gaussian quasi-Bayesian information criterion (dpGQBIC) as
\begin{align*}
\betabic_{n}^{\sharp}=-2\mbbh_{n}\big(\tes(\lambda);\lambda\big)+p\log n+2\frac{nh^{d\lambda/2}}{\lambda}-2n.
\end{align*}
To select the optimal coefficient among the candidate models using $\betabic_{n}^{\sharp}$, we compute it for each candidate model with $\lambda$ fixed.
If the same value of $\lambda$ is used for all candidate models when computing $\betabic_{n}^{\sharp}$, the third and fourth terms of $\betabic_{n}^{\sharp}$ are deterministic and common to all candidate models, and hence can be omitted when comparing the values of $\betabic_{n}^{\sharp}$.
This leads to the following simplified criterion, which we refer to as the dpGQBIC:
\begin{align*}
\betabic_{n}=-2\mbbh_{n}\big(\tes(\lambda);\lambda\big)+p\log n.
\end{align*}
Analogously to the dpGQBIC, the H\"{o}lder-based Gaussian quasi-Bayesian information criterion (HGQBIC) is defined as
\begin{align*}
\gambic_{n}=-2\mbbh_{n}^{\flat}\big(\hat{\theta}_{n}^{\flat}(\lambda);\lambda\big)+p\log n.
\end{align*}
Based on the dpGQBIC and HGQBIC, we select the optimal coefficients $\sigma_{\hat{m}_{n}(\lambda)}$ and $\sigma_{\hat{m}_{n}^{\flat}(\lambda)}$ among the candidates by
\begin{align*}
\{\hat{m}_{n}(\lambda)\}
&=\argmin_{m\in\{1,\ldots,M\}} \betabic^{(m)}_{n}, \\
\{\hat{m}_{n}^{\flat}(\lambda)\}
&=\argmin_{m\in\{1,\ldots,M\}} \gambic^{(m)}_{n},
\end{align*}
respectively.
Here, $\betabic^{(m)}_{n}$ and $\gambic^{(m)}_{n}$ denote the dpGQBIC and HGQBIC of $m$-th candidate model, respectively.



\subsection{Asymptotic probability of relative model selection}
In this section, we assume that the candidate coefficients $\sigma_{1},\ldots,\sigma_{M}$ contain both correctly specified coefficients and misspecified coefficients.
We formally use the dpGQBIC and HGQBIC even for the possibly misspecified coefficients.
Moreover, we denote $\mathbb{H}_{n}(\theta;\lambda)$ and $\mathbb{H}_{n}^{\flat}(\theta;\lambda)$ by $\mathbb{H}_{m,n}(\theta_{m};\lambda)$ and $\mathbb{H}_{m,n}^{\flat}(\theta_{m};\lambda)$ in the $m$-th candidate model $\mathcal{M}_{m}$.

Let $\mathfrak{M}$ denote the set of correctly specified models:
\begin{align*}
\mathfrak{M}=\left\{m\in\{1,\ldots,M\}: \text{there exists a } \theta_{m,0}\in\Theta_{m} \text{ such that } S_{m}(\cdot,\theta_{m,0})=S(\cdot)\right\},
\end{align*}
where $S_{m}(x,\theta_{m})=\sigma_{m}^{\otimes 2}(x,\theta_{m})$ and $S(x)=\sigma^{\otimes 2}(x)$.
We assume that the model index $m^{\ast}$ is uniquely determined
\begin{align*}
\{m^{\ast}\}=\argmin_{m\in\mathfrak{M}}\dim(\Theta_{m}).
\end{align*}

For any $m\in\{1,\ldots,M\}$, define 
\begin{align*}
\mathbb{Y}_{m,0}(\theta_{m})
&:=-\frac{1}{2T}\int_{0}^{T}\bigg\{\log\left(\frac{\det S_{m}(X_{t},\theta_{m})}{\det S(X_{t})}\right) \\
&\hspace{25mm}+\trace\left(S_{m}(X_{t},\theta_{m})^{-1}S(X_{t})-I_{d}\right)\bigg\}dt.
\end{align*}
If $m_{t}\in\mathfrak{M}$, then
\begin{align*}
\mathbb{Y}_{m_{t},n}(\theta_{m_{t}};\lambda)
&:=\frac{1}{n}\big(\mbbh_{m_{t},n}(\theta_{m_{t}};\lambda)-\mbbh_{m_{t},n}(\theta_{m_{t},0};\lambda)\big)\cip\mathbb{Y}_{m_{t},0}(\theta_{m_{t}}), \\
\mathbb{Y}_{m_{t},n}^{\flat}(\theta_{m_{t}};\lambda)
&:=\frac{1}{n}\big(\mbbh_{m_{t},n}^{\flat}(\theta_{m_{t}};\lambda)-\mbbh_{m_{t},n}^{\flat}(\theta_{m_{t},0};\lambda)\big)\cip\mathbb{Y}_{m_{t},0}(\theta_{m_{t}})
\end{align*}
uniformly in $\theta_{m_{t}}$ (see Lemma \ref{se:lem1}), and the true parameter $\theta_{m_{t},0}$ satisfies
\begin{align*}
\{\theta_{m_{t},0}\}=\argmax_{\theta_{m_{t}}\in\bar{\Theta}_{m_{t}}}\mathbb{Y}_{m_{t},0}(\theta_{m_{t}}).
\end{align*}
Furthermore, for any $m_{1},m_{2}\in\mathfrak{M}$, the equality $S_{m_{1}}(\cdot,\theta_{m_{1},0})=S_{m_{2}}(\cdot,\theta_{m_{2},0})=S(\cdot)$ implies $\mathbb{Y}_{m_{1},0}(\theta_{m_{1},0})=\mathbb{Y}_{m_{2},0}(\theta_{m_{2},0})$.

Let $\Theta_{m_{1}}\subset\mathbb{R}^{p_{m_{1}}}$ and $\Theta_{m_{2}}\subset\mathbb{R}^{p_{m_{2}}}$ be the parameter spaces associated with $\mathcal{M}_{m_{1}}$ and $\mathcal{M}_{m_{2}}$, respectively.
We say that $\Theta_{m_{1}}$ is {\it nested} in $\Theta_{m_{2}}$ when $p_{m_{1}}<p_{m_{2}}$ and there exists a matrix $F\in\mathbb{R}^{p_{m_{2}}\times p_{m_{1}}}$ with $F^{\top}F=I_{p_{m_{1}}}$ and a constant $c\in\mathbb{R}^{p_{m_{2}}}$ such that $S_{m_{1}}(\cdot,\theta_{m_{1}})=S_{m_{2}}(\cdot,F\theta_{m_{1}}+c)$ for all $\theta_{m_{1}}\in\Theta_{m_{1}}$.

Let $\mathfrak{M}^{c}=\{1,\ldots,M\}\setminus\mathfrak{M}$ denote the set of indices of misspecified models.
The following assumption is required to derive an inequality relationship between the information criteria for the true model and the misspecified model. 
\begin{ass}
\label{se:ass_modconsis}
For any $m_{c}\in\mathfrak{M}^{c}$, we have either
\begin{itemize}
\item[(i)] $\sup_{\theta_{m_{c}}\in\bar{\Theta}_{m_{c}}}\mathbb{Y}_{m_{c},0}(\theta_{m_{c}}) < 0$ a.s., or

\item[(ii)] there exists $\bar{\theta}_{m_{c}}\in\Theta_{m_{c}}$ such that $\hat{\theta}_{m_{c},n}(\lambda)\cip\bar{\theta}_{m_{c}}$, $\hat{\theta}_{m_{c},n}^{\flat}(\lambda)\cip\bar{\theta}_{m_{c}}$ as $n\to\infty$, and $\mathbb{Y}_{m_{c},0}(\bar{\theta}_{m_{c}}) < 0$ a.s.
\end{itemize}
\end{ass}




\begin{thm} \label{se:thm_consis}
Suppose that Assumptions \ref{hm:A_diff.coeff}--\ref{hm:A_lam} and \ref{se:prior} hold for all candidate coefficients which are included in $\mathfrak{M}$.

\begin{itemize}
\item[(i)] Let $m\in\mathfrak{M}\setminus\{m^{\ast}\}$. 
If $\Theta_{m^{\ast}}$ is nested in $\Theta_{m}$, then
\begin{align}
&\lim_{n\to\infty}P\left[\betabic^{(m)}_{n}>\betabic^{(m^{\ast})}_{n}\right]=1, \label{se:model_consis1} \\
&\lim_{n\to\infty}P\left[\gambic^{(m)}_{n}>\gambic^{(m^{\ast})}_{n}\right]=1. \label{se:model_consis2}
\end{align}

\item[(ii)] If Assumption \ref{se:ass_modconsis} holds, then
\begin{align}
&\lim_{n\to\infty}P\left[\min_{m_{c}\in\mathfrak{M}^{c}}\betabic^{(m_{c})}_{n}>\betabic^{(m^{\ast})}_{n}\right]=1, \label{se:model_consis3} \\
&\lim_{n\to\infty}P\left[\min_{m_{c}\in\mathfrak{M}^{c}}\gambic^{(m_{c})}_{n}>\gambic^{(m^{\ast})}_{n}\right]=1. \label{se:model_consis4}
\end{align}
\end{itemize}
\end{thm}

Theorem \ref{se:thm_consis} (i) shows that the probability of selecting the correctly specified model with the smallest dimension converges to $1$ as $n\to\infty$.
Moreover, Theorem \ref{se:thm_consis} (ii) implies that the probability of choosing any misspecified model converges to $0$ as $n\to\infty$.

\begin{rem}
\label{hm:rem_fixed.lam}
We are considering dpGQBIC and HGQBIC under the condition $\lam_n\to 0$ (see Assumption \ref{hm:A_lam}).
However, as in \cite{EguMas25}, it is possible to show that the similar claims of Theorems \ref{se:thm_denbic} and \ref{se:thm_consis} remain valid even when $\lam$ is a fixed positive constant. 
\end{rem}

\section{Numerical experiments} \label{sec:simu}

In this section, we present simulation results to observe the finite-sample performances of the density-power GQBIC and H\"{o}lder-based GQBIC.
We use the {\tt yuima} package in R (see \cite{YUIMA14}) to generate data.
All Monte Carlo trials are based on 1000 independent sample paths, and simulations are done for $\lambda=0.01$, $0.05$, $0.2$, and $n=100$, $500$, $1000$ with $T=1$.
In the following simulations, $w$ is a one-dimensional standard Wiener process, $J$ is a compound Poisson process with intensity $q$, and the distribution representing the jump-size of the compound Poisson process is given by $N(0,3)$.
Moreover, we compare the model selection frequencies through dpGQBIC, HGQBIC, and GQBIC.
The dpGQBIC, HGQBIC, and GQBIC of the $m$-th candidate model are given by
\begin{align*}
\mathrm{dpGQBIC}^{(m)}_{n}&=-2\mbbh_{m,n}\big(\hat{\theta}_{m,n}(\lambda);\lambda\big)+p_{m}\log n, \\
\mathrm{HGQBIC}^{(m)}_{n}&=-2\mbbh_{m,n}^{\flat}
\big(\hat{\theta}_{m,n}^{\flat}(\lambda);\lambda\big)+p_{m}\log n, \\
\mathrm{GQBIC}^{(m)}_{n}&=-2\mbbh_{m,n}(\hat{\theta}_{m,n})+p_{m}\log n,
\end{align*}
respectively.
Here, $\mbbh_{m,n}(\theta_{m})$ and $\hat{\theta}_{m,n}$ are the GQLF and GQMLE of $m$-th candidate model, respectively. 

\subsection{Time-inhomogeneous Wiener process} \label{se:simu1}
Let $(X_{t_{j}},Y_{t_{j}})_{j=0}^{n}$ be a data set with $t_{j}=j/n$.
Suppose that we have the sample data $(X_{t_{j}},Y_{t_{j}})_{j=0}^{n}$ from the true model
\begin{align*}
dY_{t}&=\exp\left\{\frac{1}{2}X_{t}\left(\begin{array}{c} -2 \\ 3 \\ 0 \end{array}\right)\right\}dw_{t}+dJ_{t}=\exp\left\{\frac{1}{2}(-2X_{1,t}+3X_{2,t})\right\}dw_{t}+dJ_{t}, \\
X_{t_{j}}&=(X_{1,t_{j}},X_{2,t_{j}},X_{3,t_{j}})=\left(\cos\left(\frac{2j\pi}{n}\right),\sin\left(\frac{2j\pi}{n}\right),\cos\left(\frac{4j\pi}{n}\right)\right), \\
Y_{0}&=0, \quad t\in[0,1].
\end{align*}
The simulations are performed for $q=0.01n$ and $0.1n$, corresponding to cases where the jump intensity increases with the sample size $n$, and for $q=10$, corresponding to the case where the jump intensity does not depend on $n$.
Moreover, the simulation result for $q=0$, corresponding to the case of no jumps, is provided in Appendix \ref{se:app2}.
Figure \ref{pathplot1} shows one of 1000 sample paths for each sample size in $q=0.01n$.
We consider the following candidate diffusion coefficients:
\begin{align*}
&\textbf{Diff 1}: \; \sigma_{1}(x,\theta_{1})=\exp\left\{\frac{1}{2}(\theta_{11}X_{1,t}+\theta_{12}X_{2,t}+\theta_{13}X_{3,t})\right\}; \\
&\textbf{Diff 2}: \; \sigma_{2}(x,\theta_{2})=\exp\left\{\frac{1}{2}(\theta_{21}X_{1,t}+\theta_{22}X_{2,t})\right\}; \\
&\textbf{Diff 3}: \; \sigma_{3}(x,\theta_{3})=\exp\left\{\frac{1}{2}(\theta_{31}X_{1,t}+\theta_{33}X_{3,t})\right\}; \\
&\textbf{Diff 4}: \; \sigma_{4}(x,\theta_{4})=\exp\left\{\frac{1}{2}(\theta_{42}X_{2,t}+\theta_{43}X_{3,t})\right\}; \\
&\textbf{Diff 5}: \; \sigma_{5}(x,\theta_{5})=\exp\left\{\frac{1}{2}(\theta_{51}X_{1,t})\right\}; \\
&\textbf{Diff 6}: \; \sigma_{6}(x,\theta_{6})=\exp\left\{\frac{1}{2}(\theta_{62}X_{2,t})\right\}; \\
&\textbf{Diff 7}: \; \sigma_{7}(x,\theta_{7})=\exp\left\{\frac{1}{2}(\theta_{73}X_{3,t})\right\}.
\end{align*}
The $m$-th candidate model is described by
\begin{align*}
dY_{t}=\sigma_{m}(X_{t},\theta_{m})dw_{t}.
\end{align*}
The true coefficient corresponds to Diff 2 with $(\theta_{21},\theta_{22})=(-2,3)$, and Diff 1 contains the true coefficient.

Let $\theta_{2,0}=(\theta_{21,0},\theta_{22,0})=(-2,3)$. Figures \ref{boxplot1}-- \ref{boxplot13} show the boxplots of $\hat{\theta}_{2,n}(\lambda)-\theta_{2,0}$ and $\hat{\theta}_{2,n}^{\flat}(\lambda)-\theta_{2,0}$ for each $\lambda$ with $n=500$ in Diff 2. 
The estimators $\hat{\theta}_{2,n}(0)$ and $\hat{\theta}_{2,n}^{\flat}(0)$ mean the GQMLE $\hat{\theta}_{2,n}$.
From these figures, both density-power and H\"{o}lder-based GQMLEs with $\lambda=0.2$ perform better than those with other values of $\lambda$.

Tables \ref{candsimu1}--\ref{candsimu13} summarize the model selection frequencies.
The GQBIC frequently selects Diff 1, which corresponds to a larger coefficient than the true one, and the selection frequency of Diff 2 under the GQBIC does not increase with $n$. 
That is, these results do not support the model selection consistency of the GQBIC.
The selection frequency of Diff 1 under the dpGQBIC increases as $\lambda$ decreases; on the other hand, the selection frequency of Diff 6, which corresponds to a smaller coefficient than the true one, under the HGQBIC increases as $\lambda$ decreases.
In the cases where $q=0.01n$ and $0.1n$, both the dpGQBIC and the HGQBIC tend to select Diff 2 more frequently as $n$ increases for $\lambda=0.05$ and $0.2$.
Moreover, the HGQBIC shows the same tendency for $\lambda=0.01$, although the dpGQBIC selects Diff 1 with high frequency for $\lambda=0.01$.
In the case where $q=10$, the selection frequencies of Diff 2 under both criteria also increase with $n$ for all $\lambda$.

\begin{figure}[t]
\begin{tabular}{c}

\begin{minipage}{0.31 \hsize}
\begin{center}
\includegraphics[scale=0.25]{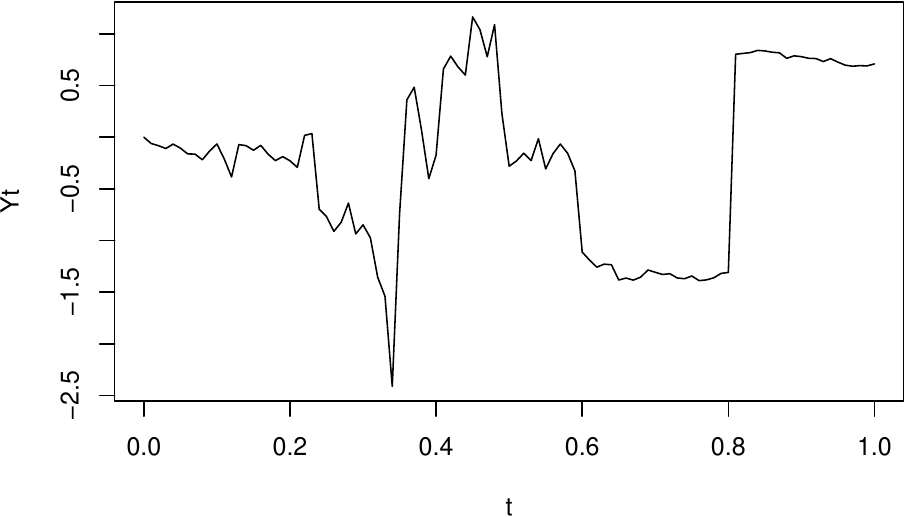}
\end{center}
\end{minipage}

\begin{minipage}{0.31 \hsize}
\begin{center}
\includegraphics[scale=0.25]{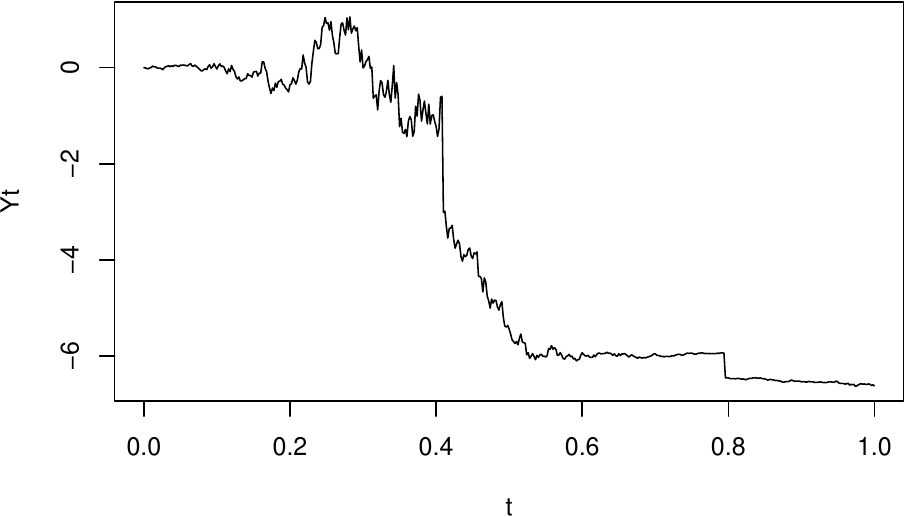}
\end{center}
\end{minipage}

\begin{minipage}{0.31 \hsize}
\begin{center}
\includegraphics[scale=0.25]{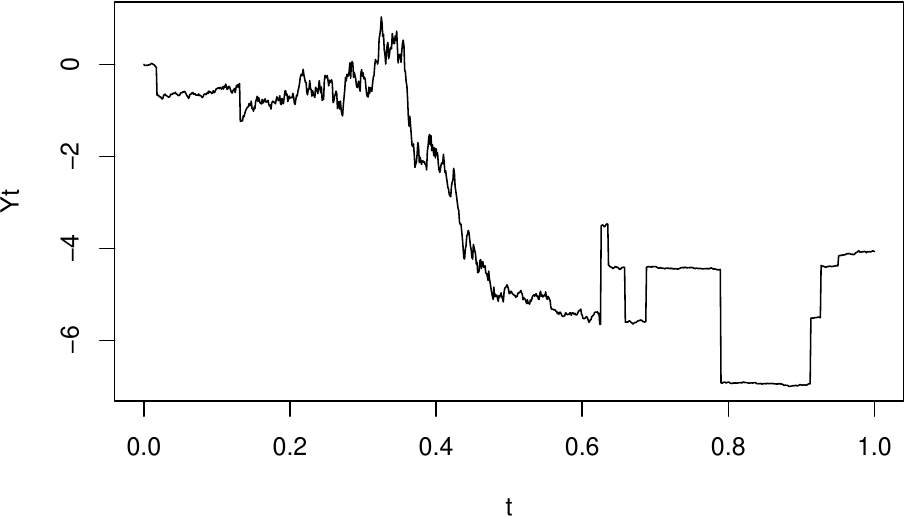}
\end{center}
\end{minipage}

\end{tabular}
\caption{One of 1000 sample paths in Section \ref{se:simu1} (left: $n=100$, center: $n=500$, right: $n=1000$).}
\label{pathplot1}
\end{figure}

\begin{figure}[t]
\begin{tabular}{c}

\begin{minipage}{0.45 \hsize}
\begin{center}
\includegraphics[scale=0.38]{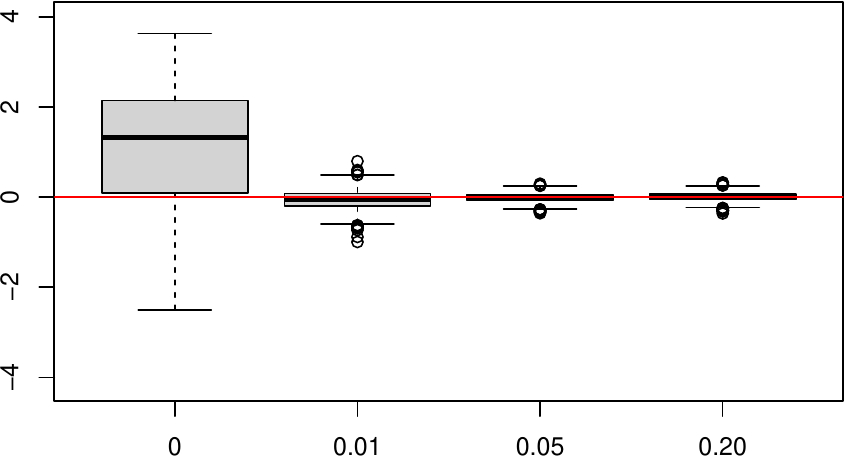}
\ \\
$\hat{\theta}_{21,n}(\lambda)-\theta_{21,0}$
\end{center}
\end{minipage}

\begin{minipage}{0.45 \hsize}
\begin{center}
\includegraphics[scale=0.38]{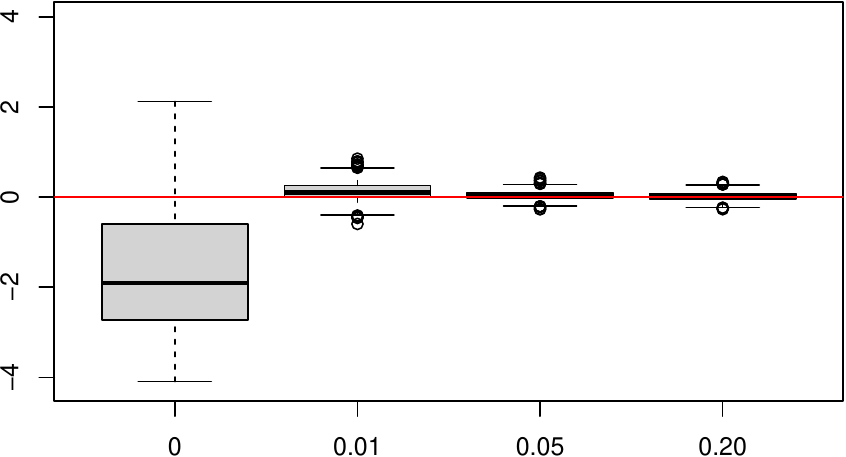}
\ \\
$\hat{\theta}_{22,n}(\lambda)-\theta_{22,0}$
\end{center}
\end{minipage}

\ \\
\ \\

\begin{minipage}{0.45 \hsize}
\begin{center}
\includegraphics[scale=0.38]{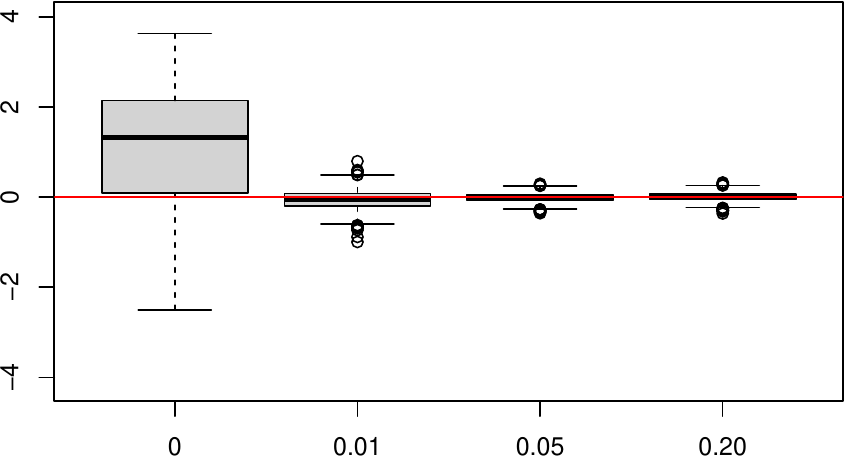}
\ \\
$\hat{\theta}_{21,n}^{\flat}(\lambda)-\theta_{21,0}$
\end{center}
\end{minipage}

\begin{minipage}{0.45 \hsize}
\begin{center}
\includegraphics[scale=0.38]{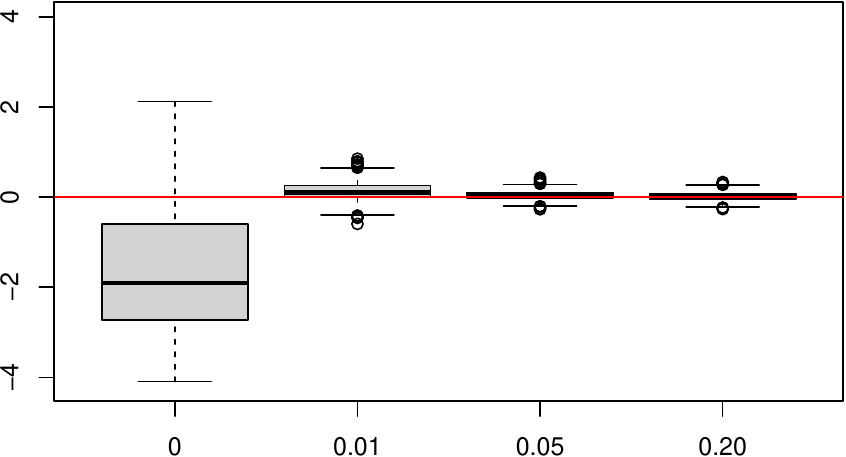}
\ \\
$\hat{\theta}_{22,n}^{\flat}(\lambda)-\theta_{22,0}$
\end{center}
\end{minipage}

\end{tabular}
\caption{Boxplots of $\hat{\theta}_{2,n}(\lambda)-\theta_{2,0}$ and $\hat{\theta}_{2,n}^{\flat}(\lambda)-\theta_{2,0}$ for each $\lambda$ when the candidate coefficient is Diff 2 in Section \ref{se:simu1} ($q=0.01n$, $n=500$).}
\label{boxplot1}
\end{figure}


\begin{figure}[t]
\begin{tabular}{c}

\begin{minipage}{0.45 \hsize}
\begin{center}
\includegraphics[scale=0.38]{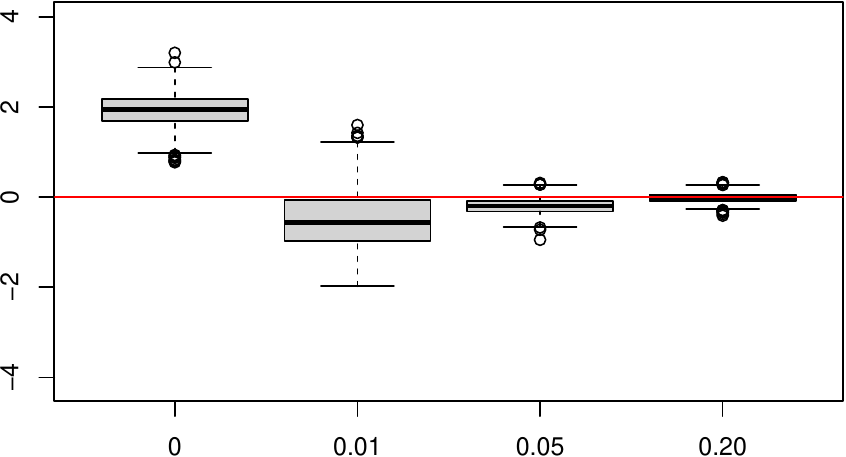}
\ \\
$\hat{\theta}_{21,n}(\lambda)-\theta_{21,0}$
\end{center}
\end{minipage}

\begin{minipage}{0.45 \hsize}
\begin{center}
\includegraphics[scale=0.38]{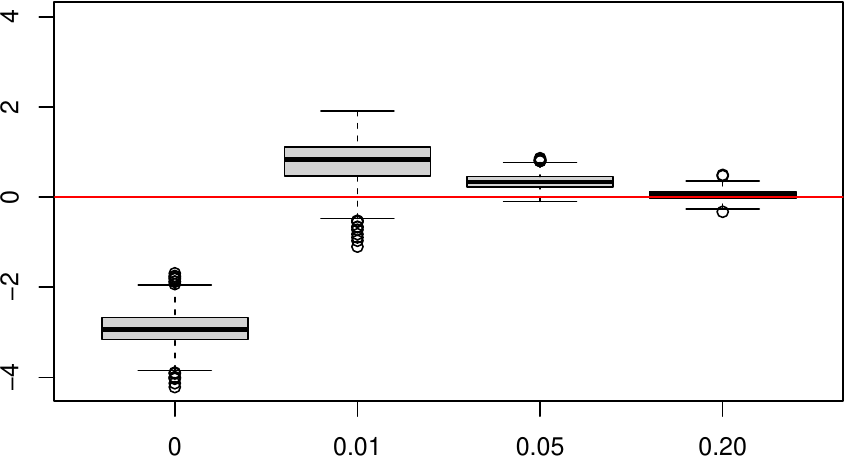}
\ \\
$\hat{\theta}_{22,n}(\lambda)-\theta_{22,0}$
\end{center}
\end{minipage}

\ \\
\ \\

\begin{minipage}{0.45 \hsize}
\begin{center}
\includegraphics[scale=0.38]{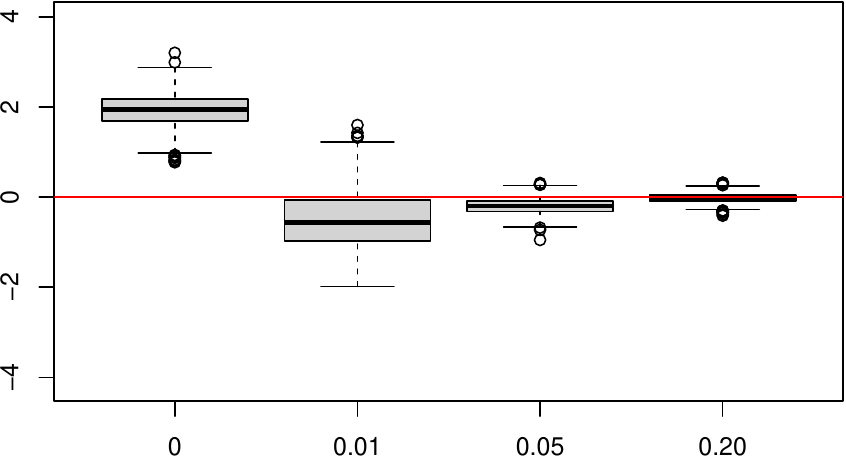}
\ \\
$\hat{\theta}_{21,n}^{\flat}(\lambda)-\theta_{21,0}$
\end{center}
\end{minipage}

\begin{minipage}{0.45 \hsize}
\begin{center}
\includegraphics[scale=0.38]{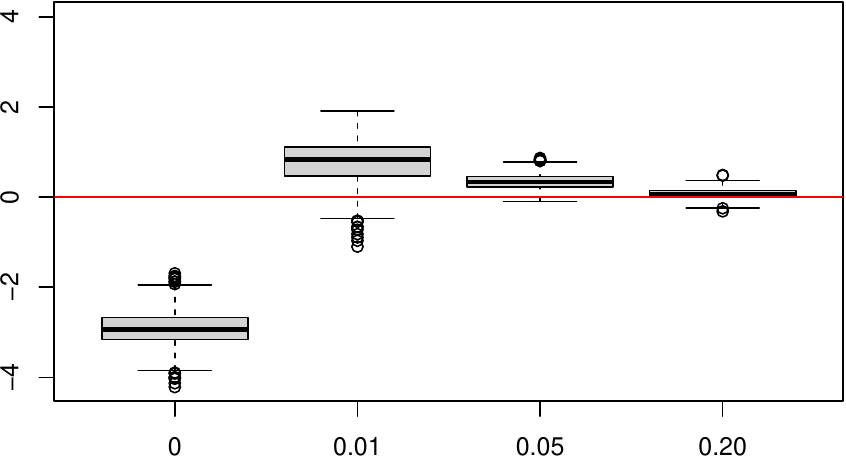}
\ \\
$\hat{\theta}_{22,n}^{\flat}(\lambda)-\theta_{22,0}$
\end{center}
\end{minipage}

\end{tabular}
\caption{Boxplots of $\hat{\theta}_{2,n}(\lambda)-\theta_{2,0}$ and $\hat{\theta}_{2,n}^{\flat}(\lambda)-\theta_{2,0}$ for each $\lambda$ when the candidate coefficient is Diff 2 in Section \ref{se:simu1} ($q=0.1n$, $n=500$).}
\label{boxplot12}
\end{figure}

\begin{figure}[t]
\begin{tabular}{c}

\begin{minipage}{0.45 \hsize}
\begin{center}
\includegraphics[scale=0.38]{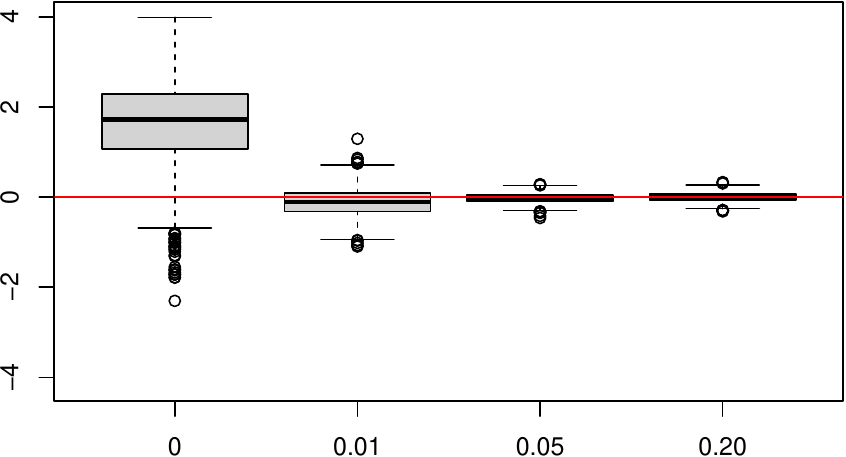}
\ \\
$\hat{\theta}_{21,n}(\lambda)-\theta_{21,0}$
\end{center}
\end{minipage}

\begin{minipage}{0.45 \hsize}
\begin{center}
\includegraphics[scale=0.38]{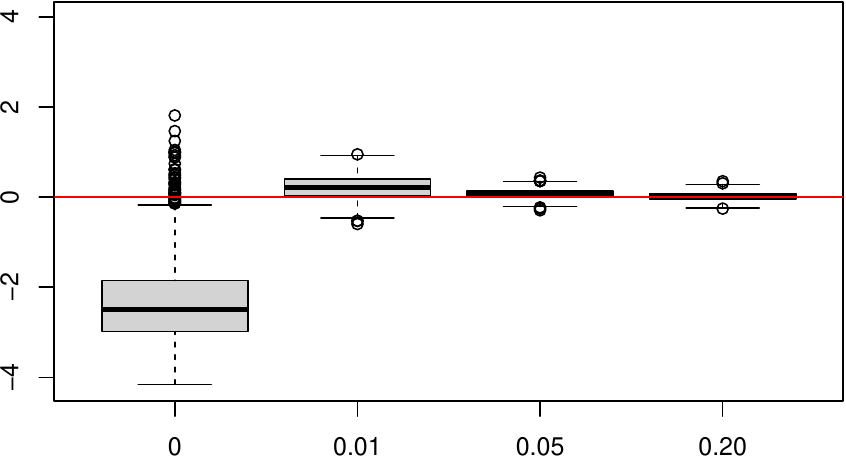}
\ \\
$\hat{\theta}_{22,n}(\lambda)-\theta_{22,0}$
\end{center}
\end{minipage}

\ \\
\ \\

\begin{minipage}{0.45 \hsize}
\begin{center}
\includegraphics[scale=0.38]{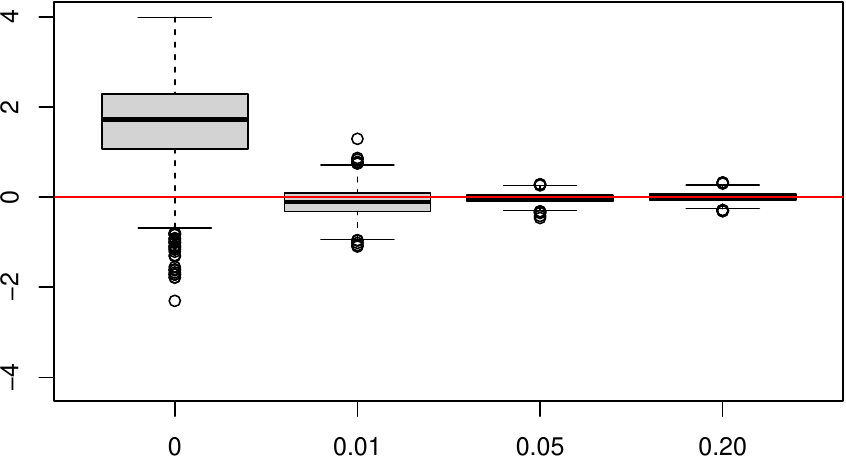}
\ \\
$\hat{\theta}_{21,n}^{\flat}(\lambda)-\theta_{21,0}$
\end{center}
\end{minipage}

\begin{minipage}{0.45 \hsize}
\begin{center}
\includegraphics[scale=0.38]{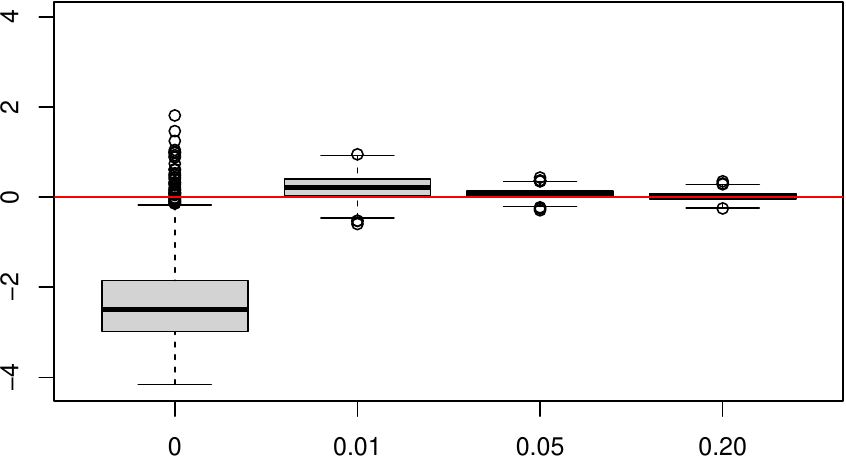}
\ \\
$\hat{\theta}_{22,n}^{\flat}(\lambda)-\theta_{22,0}$
\end{center}
\end{minipage}

\end{tabular}
\caption{Boxplots of $\hat{\theta}_{2,n}(\lambda)-\theta_{2,0}$ and $\hat{\theta}_{2,n}^{\flat}(\lambda)-\theta_{2,0}$ for each $\lambda$ when the candidate coefficient is Diff 2 in Section \ref{se:simu1} ($q=10$, $n=500$).}
\label{boxplot13}
\end{figure}

\begin{table}[t]
\begin{center}
\caption{Model selection frequencies for various situations in Section \ref{se:simu1} ($q=0.01n$).}
\scalebox{0.9}[0.9]{
\begin{tabular}{l l | r r r r r r r} \hline
& & \multicolumn{7}{c}{} \\[-3mm]
dpGQBIC & $n$ & Diff 1 & Diff $2^{\ast}$ & Diff 3 & Diff 4 & Diff 5 & Diff 6 & Diff 7 \\[1mm] \hline
& & \multicolumn{7}{c}{} \\[-3mm]
$\lambda=0.01$ & $100$ & 239 & 742 & 0 & 19 & 0 & 0 & 0 \\[1mm]
& $500$ & 304 & 696 & 0 & 0 & 0 & 0 & 0 \\[1mm]
& $1000$ & 332 & 668 & 0 & 0 & 0 & 0 & 0 \\[1mm] \hline
& & \multicolumn{7}{c}{} \\[-3mm]
$\lambda=0.05$ & $100$ & 39 & 961 & 0 & 0 & 0 & 0 & 0 \\[1mm]
& $500$ & 14 & 986 & 0 & 0 & 0 & 0 & 0 \\[1mm]
& $1000$ & 3 & 997 & 0 & 0 & 0 & 0 & 0 \\[1mm] \hline
& & \multicolumn{7}{c}{} \\[-3mm]
$\lambda=0.2$ & $100$ & 6 & 994 & 0 & 0 & 0 & 0 & 0 \\[1mm]
& $500$ & 1 & 999 & 0 & 0 & 0 & 0 & 0 \\[1mm]
& $1000$ & 0 & 1000 & 0 & 0 & 0 & 0 & 0 \\[1mm] \hline
& & \multicolumn{7}{c}{} \\[-3mm]
HGQBIC & $n$ & Diff 1 & Diff $2^{\ast}$ & Diff 3 & Diff 4 & Diff 5 & Diff 6 & Diff 7 \\[1mm] \hline
& & \multicolumn{7}{c}{} \\[-3mm]
$\lambda=0.01$ & $100$ & 0 & 0 & 0 & 0 & 26 & 974 & 0 \\[1mm]
& $500$ & 0 & 471 & 0 & 0 & 0 & 529 & 0 \\[1mm]
& $1000$ & 0 & 999 & 0 & 0 & 0 & 1 & 0 \\[1mm] \hline
& & \multicolumn{7}{c}{} \\[-3mm]
$\lambda=0.05$ & $100$ & 0 & 425 & 0 & 0 & 0 & 575 & 0 \\[1mm]
& $500$ & 0 & 1000 & 0 & 0 & 0 & 0 & 0 \\[1mm]
& $1000$ & 0 & 1000 & 0 & 0 & 0 & 0 & 0 \\[1mm] \hline
& & \multicolumn{7}{c}{} \\[-3mm]
$\lambda=0.2$ & $100$ & 0 & 992 & 0 & 0 & 0 & 8 & 0 \\[1mm]
& $500$ & 0 & 1000 & 0 & 0 & 0 & 0 & 0 \\[1mm]
& $1000$ & 0 & 1000 & 0 & 0 & 0 & 0 & 0 \\[1mm] \hline
& & \multicolumn{7}{c}{} \\[-3mm]
GQBIC & $n$ & Diff 1 & Diff $2^{\ast}$ & Diff 3 & Diff 4 & Diff 5 & Diff 6 & Diff 7 \\[1mm] \hline
& & \multicolumn{7}{c}{} \\[-3mm]
& $100$ & 352 & 597 & 15 & 33 & 1 & 1 & 1 \\[1mm]
& $500$ & 840 & 94 & 27 & 33 & 2 & 3 & 1 \\[1mm]
& $1000$ & 908 & 40 & 23 & 26 & 1 & 2 & 0 \\[1mm] \hline
\end{tabular}
}
\label{candsimu1}
\end{center}
\end{table}

\begin{table}[t]
\begin{center}
\caption{Model selection frequencies for various situations in Section \ref{se:simu1} ($q=0.1n$).}
\scalebox{0.9}[0.9]{
\begin{tabular}{l l | r r r r r r r} \hline
& & \multicolumn{7}{c}{} \\[-3mm]
dpGQBIC & $n$ & Diff 1 & Diff $2^{\ast}$ & Diff 3 & Diff 4 & Diff 5 & Diff 6 & Diff 7 \\[1mm] \hline
& & \multicolumn{7}{c}{} \\[-3mm]
$\lambda=0.01$ & $100$ & 626 & 208 & 29 & 133 & 0 & 4 & 0 \\[1mm]
& $500$ & 702 & 287 & 0 & 11 & 0 & 0 & 0 \\[1mm]
& $1000$ & 725 & 275 & 0 & 0 & 0 & 0 & 0 \\[1mm] \hline
& & \multicolumn{7}{c}{} \\[-3mm]
$\lambda=0.05$ & $100$ & 235 & 764 & 0 & 1 & 0 & 0 & 0 \\[1mm]
& $500$ & 218 & 782 & 0 & 0 & 0 & 0 & 0 \\[1mm]
& $1000$ & 190 & 810 & 0 & 0 & 0 & 0 & 0 \\[1mm] \hline
& & \multicolumn{7}{c}{} \\[-3mm]
$\lambda=0.2$ & $100$ & 16 & 984 & 0 & 0 & 0 & 0 & 0 \\[1mm]
& $500$ & 6 & 994 & 0 & 0 & 0 & 0 & 0 \\[1mm]
& $1000$ & 0 & 1000 & 0 & 0 & 0 & 0 & 0 \\[1mm] \hline
& & \multicolumn{7}{c}{} \\[-3mm]
HGQBIC & $n$ & Diff 1 & Diff $2^{\ast}$ & Diff 3 & Diff 4 & Diff 5 & Diff 6 & Diff 7 \\[1mm] \hline
& & \multicolumn{7}{c}{} \\[-3mm]
$\lambda=0.01$ & $100$ & 0 & 10 & 1 & 1 & 176 & 799 & 13 \\[1mm]
& $500$ & 0 & 669 & 0 & 4 & 0 & 327 & 0 \\[1mm]
& $1000$ & 0 & 996 & 0 & 0 & 0 & 4 & 0 \\[1mm] \hline
& & \multicolumn{7}{c}{} \\[-3mm]
$\lambda=0.05$ & $100$ & 0 & 466 & 0 & 0 & 0 & 534 & 0 \\[1mm]
& $500$ & 0 & 1000 & 0 & 0 & 0 & 0 & 0 \\[1mm]
& $1000$ & 0 & 1000 & 0 & 0 & 0 & 0 & 0 \\[1mm] \hline
& & \multicolumn{7}{c}{} \\[-3mm]
$\lambda=0.2$ & $100$ & 0 & 970 & 0 & 0 & 0 & 30 & 0 \\[1mm]
& $500$ & 0 & 1000 & 0 & 0 & 0 & 0 & 0 \\[1mm]
& $1000$ & 0 & 1000 & 0 & 0 & 0 & 0 & 0 \\[1mm] \hline
& & \multicolumn{7}{c}{} \\[-3mm]
GQBIC & $n$ & Diff 1 & Diff $2^{\ast}$ & Diff 3 & Diff 4 & Diff 5 & Diff 6 & Diff 7 \\[1mm] \hline
& & \multicolumn{7}{c}{} \\[-3mm]
& $100$ & 780 & 77 & 57 & 68 & 4 & 8 & 6 \\[1mm]
& $500$ & 872 & 36 & 45 & 41 & 1 & 3 & 2 \\[1mm]
& $1000$ & 913 & 29 & 32 & 22 & 0 & 3 & 1 \\[1mm] \hline
\end{tabular}
}
\label{candsimu12}
\end{center}
\end{table}

\begin{table}[t]
\begin{center}
\caption{Model selection frequencies for various situations in Section \ref{se:simu1} ($q=10$).}
\scalebox{0.9}[0.9]{
\begin{tabular}{l l | r r r r r r r} \hline
& & \multicolumn{7}{c}{} \\[-3mm]
dpGQBIC & $n$ & Diff 1 & Diff $2^{\ast}$ & Diff 3 & Diff 4 & Diff 5 & Diff 6 & Diff 7 \\[1mm] \hline
& & \multicolumn{7}{c}{} \\[-3mm]
$\lambda=0.01$ & $100$ & 626 & 208 & 29 & 133 & 0 & 4 & 0 \\[1mm]
& $500$ & 440 & 560 & 0 & 0 & 0 & 0 & 0 \\[1mm]
& $1000$ & 332 & 668 & 0 & 0 & 0 & 0 & 0 \\[1mm] \hline
& & \multicolumn{7}{c}{} \\[-3mm]
$\lambda=0.05$ & $100$ & 235 & 764 & 0 & 1 & 0 & 0 & 0 \\[1mm]
& $500$ & 24 & 976 & 0 & 0 & 0 & 0 & 0 \\[1mm]
& $1000$ & 3 & 997 & 0 & 0 & 0 & 0 & 0 \\[1mm] \hline
& & \multicolumn{7}{c}{} \\[-3mm]
$\lambda=0.2$ & $100$ & 16 & 984 & 0 & 0 & 0 & 0 & 0 \\[1mm]
& $500$ & 0 & 1000 & 0 & 0 & 0 & 0 & 0 \\[1mm]
& $1000$ & 0 & 1000 & 0 & 0 & 0 & 0 & 0 \\[1mm] \hline
& & \multicolumn{7}{c}{} \\[-3mm]
HGQBIC & $n$ & Diff 1 & Diff $2^{\ast}$ & Diff 3 & Diff 4 & Diff 5 & Diff 6 & Diff 7 \\[1mm] \hline
& & \multicolumn{7}{c}{} \\[-3mm]
$\lambda=0.01$ & $100$ & 0 & 10 & 1 & 1 & 176 & 799 & 13 \\[1mm]
& $500$ & 0 & 480 & 0 & 0 & 0 & 520 & 0 \\[1mm]
& $1000$ & 0 & 999 & 0 & 0 & 0 & 1 & 0 \\[1mm] \hline
& & \multicolumn{7}{c}{} \\[-3mm]
$\lambda=0.05$ & $100$ & 0 & 466 & 0 & 0 & 0 & 534 & 0 \\[1mm]
& $500$ & 0 & 1000 & 0 & 0 & 0 & 0 & 0 \\[1mm]
& $1000$ & 0 & 1000 & 0 & 0 & 0 & 0 & 0 \\[1mm] \hline
& & \multicolumn{7}{c}{} \\[-3mm]
$\lambda=0.2$ & $100$ & 0 & 970 & 0 & 0 & 0 & 30 & 0 \\[1mm]
& $500$ & 0 & 1000 & 0 & 0 & 0 & 0 & 0 \\[1mm]
& $1000$ & 0 & 1000 & 0 & 0 & 0 & 0 & 0 \\[1mm] \hline
& & \multicolumn{7}{c}{} \\[-3mm]
GQBIC & $n$ & Diff 1 & Diff $2^{\ast}$ & Diff 3 & Diff 4 & Diff 5 & Diff 6 & Diff 7 \\[1mm] \hline
& & \multicolumn{7}{c}{} \\[-3mm]
& $100$ & 780 & 77 & 57 & 68 & 4 & 8 & 6 \\[1mm]
& $500$ & 889 & 42 & 26 & 35 & 1 & 5 & 2 \\[1mm]
& $1000$ & 908 & 40 & 23 & 26 & 1 & 2 & 0 \\[1mm] \hline
\end{tabular}
}
\label{candsimu13}
\end{center}
\end{table}

\subsection{Jump-diffusion process} \label{se:simu2}
The sample data $(Y_{t_{j}})_{j=0}^{n}$ with $t_{j}=j/n$ is obtained from 
\begin{align*}
dY_{t}=Y_{t}dt+\frac{2+3Y_{t}^{2}}{1+Y_{t}^{2}}dw_{t}+dJ_{t}, \quad Y_{0}=0, \quad t\in[0,1].
\end{align*}
The simulations are performed for $q=0.01n$ and $10$.
Figure \ref{pathplot2} shows one of 1000 sample paths for each sample size in $q=0.01n$.
We consider the following candidate diffusion coefficients:
\begin{align*}
&\textbf{Diff 1}: \; \sigma_{1}(y,\theta_{1})=\frac{\theta_{11}+\theta_{12}y+\theta_{13}y^{2}}{1+y^2}; \\
&\textbf{Diff 2}: \; \sigma_{2}(y,\theta_{2})=\frac{\theta_{21}+\theta_{23}y^{2}}{1+y^2}; \\
&\textbf{Diff 3}: \; \sigma_{3}(y,\theta_{3})=\frac{\theta_{31}}{1+y^2}.
\end{align*}
The $m$-th candidate model is described by
\begin{align*}
dY_{t}=\sigma_{m}(Y_{t},\theta_{m})dw_{t}.
\end{align*}
The true coefficient corresponds to Diff 2 with $(\theta_{21},\theta_{22})=(2,3)$, and Diff 1 contains the true coefficient.

Table \ref{candsimu3} summarizes the frequency of model selection.
The GQBIC frequently selects Diff 2; however, its selection frequency for Diff 2 does not increase with $n$. 
Moreover, both the dpGQBIC and the HGQBIC exhibit the same behavior as in Section \ref{se:simu1}.

\begin{rem}
The choice of the tuning parameter $\lambda$ is important, and several studies have investigated the selection of $\lambda$ for parameter estimation; see, for example, \cite{WarJon05, BasBasJon21}.
In the literature on density-power divergence-based estimation, values around $0.2$ are often used for $\lambda$ because they provide a good balance between efficiency and robustness.
Similarly, in the numerical experiments on model selection presented in this paper, $\lambda=0.2$ yields the best performance.
Therefore, $\lambda=0.2$ may serve as a reasonable default or starting value in practice, although the optimal value may depend on the configuration considered.
\end{rem}

\begin{figure}[t]
\begin{tabular}{c}

\begin{minipage}{0.31 \hsize}
\begin{center}
\includegraphics[scale=0.25]{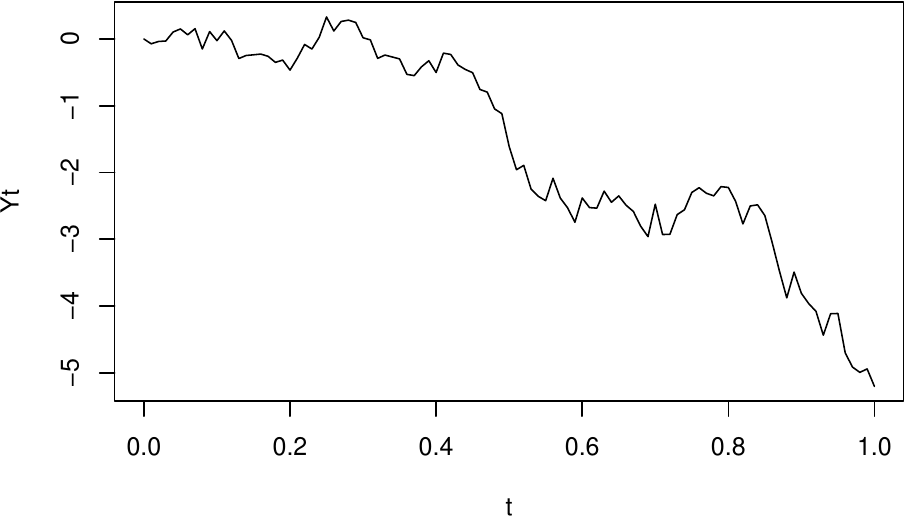}
\end{center}
\end{minipage}

\begin{minipage}{0.31 \hsize}
\begin{center}
\includegraphics[scale=0.25]{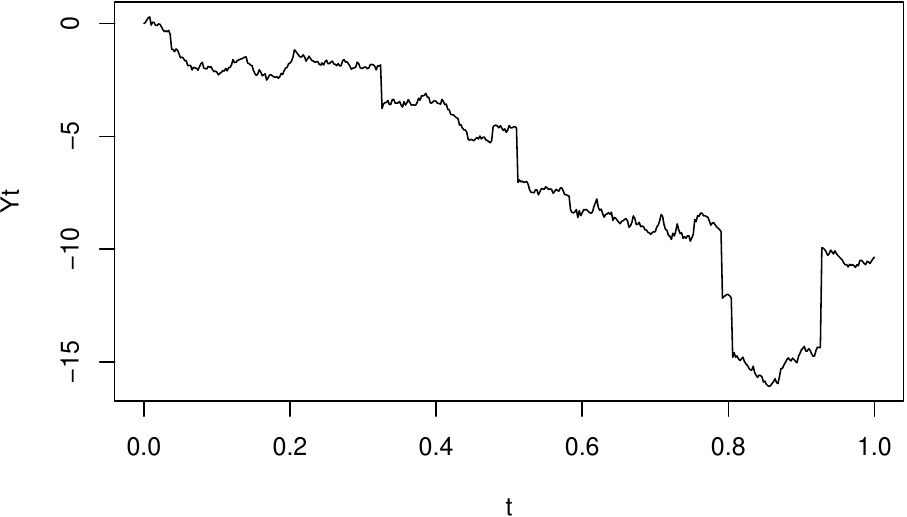}
\end{center}
\end{minipage}

\begin{minipage}{0.31 \hsize}
\begin{center}
\includegraphics[scale=0.25]{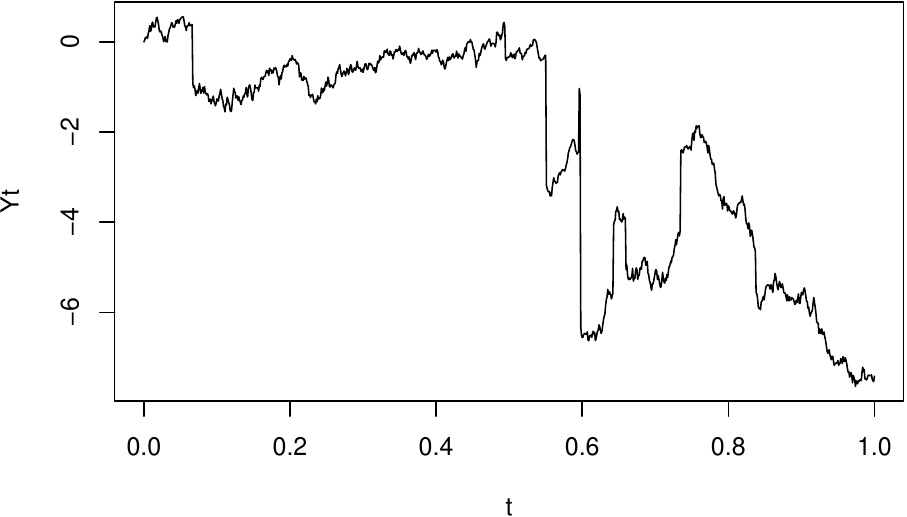}
\end{center}
\end{minipage}

\end{tabular}
\caption{One of 1000 sample paths in Section \ref{se:simu2} (left: $n=100$, center: $n=500$, right: $n=1000$).}
\label{pathplot2}
\end{figure}

\begin{table}[t]
\begin{center}
\caption{Model selection frequencies for various situations in Section \ref{se:simu2}.}
\scalebox{0.9}[0.9]{
\begin{tabular}{l l | r r r | r r r} \hline
& & \multicolumn{3}{c|}{} & \multicolumn{3}{c}{} \\[-3mm]
& & \multicolumn{3}{l|}{$q=0.01n$} & \multicolumn{3}{l}{$q=10$} \\[1mm] \hline
& & \multicolumn{3}{c|}{} & \multicolumn{3}{c}{} \\[-3mm]
dpGQBIC & $n$ & Diff 1 & Diff $2^{\ast}$ & Diff 3 & Diff 1 & Diff $2^{\ast}$ & Diff 3 \\[1mm] \hline
& & \multicolumn{3}{c|}{} & \multicolumn{3}{c}{} \\[-3mm]
$\lambda=0.01$ & $100$ & 94 & 899 & 7 & 215 & 785 & 0 \\[1mm]
& $500$ & 265 & 735 & 0 & 279 & 721 & 0 \\[1mm]
& $1000$ & 280 & 720 & 0 & 280 & 720 & 0 \\[1mm] \hline
& & \multicolumn{3}{c|}{} & \multicolumn{3}{c}{} \\[-3mm]
$\lambda=0.05$ & $100$ & 43 & 950 & 7 & 146 & 854 & 0 \\[1mm]
& $500$ & 22 & 978 & 0 & 44 & 956 & 0 \\[1mm]
& $1000$ & 8 & 992 & 0 & 8 & 992 & 0 \\[1mm] \hline
& & \multicolumn{3}{c|}{} & \multicolumn{3}{c}{} \\[-3mm]
$\lambda=0.2$ & $100$ & 3 & 987 & 10 & 10 & 900 & 0 \\[1mm]
& $500$ & 0 & 1000 & 0 & 0 & 1000 & 0 \\[1mm]
& $1000$ & 0 & 1000 & 0 & 0 & 1000 & 0 \\[1mm] \hline
& & \multicolumn{3}{c|}{} & \multicolumn{3}{c}{} \\[-3mm]
HGQBIC & $n$ & Diff 1 & Diff $2^{\ast}$ & Diff 3 & Diff 1 & Diff $2^{\ast}$ & Diff 3 \\[1mm] \hline
& & \multicolumn{3}{c|}{} & \multicolumn{3}{c}{} \\[-3mm]
$\lambda=0.01$ & $100$ & 0 & 474 & 576 & 0 & 845 & 155 \\[1mm]
& $500$ & 0 & 834 & 166 & 0 & 937 & 63 \\[1mm]
& $1000$ & 0 & 985 & 15 & 0 & 985 & 15 \\[1mm] \hline
& & \multicolumn{3}{c|}{} & \multicolumn{3}{c}{} \\[-3mm]
$\lambda=0.05$ & $100$ & 0 & 680 & 320 & 0 & 929 & 71 \\[1mm]
& $500$ & 0 & 992 & 8 & 0 & 997 & 3 \\[1mm]
& $1000$ & 0 & 1000 & 0 & 0 & 1000 & 0 \\[1mm] \hline
& & \multicolumn{3}{c|}{} & \multicolumn{3}{c}{} \\[-3mm]
$\lambda=0.2$ & $100$ & 0 & 861 & 139 & 0 & 985 & 15 \\[1mm]
& $500$ & 0 & 998 & 2 & 0 & 1000 & 0 \\[1mm]
& $1000$ & 0 & 1000 & 0 & 0 & 1000 & 0 \\[1mm] \hline
& & \multicolumn{3}{c|}{} & \multicolumn{3}{c}{} \\[-3mm]
GQBIC & $n$ & Diff 1 & Diff $2^{\ast}$ & Diff 3 & Diff 1 & Diff $2^{\ast}$ & Diff 3 \\[1mm] \hline
& & \multicolumn{3}{c|}{} & \multicolumn{3}{c}{} \\[-3mm]
& $100$ & 108 & 885 & 7 & 235 & 765  & 0 \\[1mm]
& $500$ & 324 & 676 & 0 & 334 & 666 & 0 \\[1mm]
& $1000$ & 375 & 625 & 0 & 375 & 625 & 0 \\[1mm] \hline
\end{tabular}
}
\label{candsimu3}
\end{center}
\end{table}

\section{Proofs} \label{sec:prf}

\subsection{Proof of Theorem \ref{se:thm_denbic}}
This theorem follows from \cite[Theorem 3.4]{EguMas25} and \eqref{se:thm.qbic.cip2}.
In \cite[Theorem 3.4]{EguMas25}, the asymptotic mixed normality of density-power and H\"{o}lder-based GQMLEs is established under Assumptions \ref{hm:A_diff.coeff}--\ref{hm:A_lam}.
Moreover, the proof of \cite[Theorem 3.4]{EguMas25} shows that \eqref{appendix.thm1}--\eqref{appendix.thm6} hold for the density-power and the H\"{o}lder-based GQLFs.
Therefore, \eqref{se:thm.qbic.cip} in Theorem \ref{hm:thm.qbic} applies in the present setting, and the same stochastic expansion as \eqref{se:thm.qbic.cip2} holds for the density-power and the H\"{o}lder-based GQLFs.

Similarly to \eqref{se:thm.qbic.cip2}, we have
\begin{align*}
\mathfrak{F}_{n}(h^{d\lambda/2};\lambda)
&=-\log\left[\int_{\Theta}\exp\left\{\mbbh_{n}(\theta;\lambda)\right\}\pi(\theta)d\theta\right]+\frac{nh^{d\lambda/2}}{\lambda}-n \\
&=-\mbbh_{n}(\tz;\lambda)+\frac{p}{2}\log n+\frac{nh^{d\lambda/2}}{\lambda}-n+O_{p}(1), \\
\mathfrak{F}_{n}^{\flat}(\lambda/\mathsf{k}_{\lambda};\lambda)
&=-\log\left[\int_{\Theta}\exp\left\{\mbbh_{n}^{\flat}(\theta;\lam)\right\}\pi(\theta)d\theta\right]+\frac{n}{\mathsf{k}_{\lambda}} \\
&=-\mbbh_{n}^{\flat}(\tz;\lambda)+\frac{p}{2}\log n+\frac{n}{\mathsf{k}_{\lambda}}+O_{p}(1).
\end{align*}
By \cite[Theorem 3.4]{EguMas25}, $\mbbh_{n}(\tz;\lambda)=\mbbh_{n}(\tes(\lambda);\lambda)+O_{p}(1)$ and $\mbbh_{n}^{\flat}(\tz;\lambda)=\mbbh_{n}^{\flat}(\tes^{\flat}(\lambda);\lambda)+O_{p}(1)$.
Therefore, \eqref{se:st_exp_den} and \eqref{se:st_exp_hol} are established.

\subsection{Proof of Theorem \ref{se:thm_consis}}

\subsubsection{Proof of (i)}

We prove only \eqref{se:model_consis1}, since the proof of \eqref{se:model_consis2} is similar.

Let $m\in\mathfrak{M}\setminus\{m^{\ast}\}$, and suppose that $\Theta_{m^{\ast}}$ be nested in $\Theta_{m}$ ($p_{m^{\ast}}<p_{m}$).
Define a map $f_{m}:\Theta_{m^{\ast}}\to\Theta_{m}$ by $f_{m}(\theta_{m^{\ast}})=F_{m}\theta_{m^{\ast}}+c_{m}$, where $F_{m}\in\mathbb{R}^{p_{m}\times p_{m^{\ast}}}$ and $c_{m}\in\mathbb{R}^{p_{m}}$ satisfy $S_{m^{\ast}}(\cdot,\theta_{m^{{\ast}}})=S_{m}\big(\cdot,f_{m}(\theta_{m^{\ast}})\big)$ for all $\theta_{m^{\ast}}\in\Theta_{m^{\ast}}$.
From the definition of $f_{m}$, the equations $\mathbb{H}_{m^{\ast},n}(\theta_{m^{\ast}};\lambda)=\mathbb{H}_{m,n}\big(f_{m}(\theta_{m^{\ast}});\lambda\big)$ and $\mathbb{Y}_{m^{\ast},0}(\theta_{m^{\ast}})=\mathbb{Y}_{m,0}\big(f_{m}(\theta_{m^{\ast}})\big)$ are satisfied for all $\theta_{m^{\ast}}\in\Theta_{m^{\ast}}$.
If $f_{m}(\theta_{m^{\ast},0})\neq\theta_{m,0}$, then $\mathbb{Y}_{m^{\ast},0}(\theta_{m^{\ast},0})=\mathbb{Y}_{m,0}\big(f_{m}(\theta_{m^{\ast},0})\big)<\mathbb{Y}_{m,0}(\theta_{m,0})$, which contradicts $\mathbb{Y}_{m^{\ast},0}(\theta_{m^{\ast},0})=\mathbb{Y}_{m,0}(\theta_{m,0})$.
Hence, we have $f_{m}(\theta_{m^{\ast},0})=\theta_{m,0}$.

By the Taylor expansion of $\mbbh_{m,n}$ around $\hat{\theta}_{m,n}(\lambda)$, we have
\begin{align*}
&\mbbh_{m^{\ast},n}\big(\hat{\theta}_{m^{\ast},n}(\lambda);\lambda\big) \\
&=\mbbh_{m,n}(\hat{\theta}_{m,n}(\lambda),\lambda) \\
&\qquad -\frac{1}{2}\left(-\p_{\theta_{m}}^{2}\mbbh_{m,n}(\tilde{\theta}_{m,n}(\lambda);\lambda)\right)\left[\left(\hat{\theta}_{m,n}(\lambda)-f_{m}(\hat{\theta}_{m^{\ast},n}(\lambda))\right)^{\otimes2}\right],
\end{align*}
where $\tilde{\theta}_{m,n}(\lambda)\cip\theta_{m,0}$ as $n\to\infty$.
Moreover, from \cite[Theorem 3.4 and Lemma 4]{EguMas25}, we have 
\begin{align*}
\hat{\theta}_{m,n}(\lambda)-f_{m}(\hat{\theta}_{m^{\ast},n}(\lambda))
&=(\hat{\theta}_{m,n}(\lambda)-\theta_{m,0})-\big(f_{m}(\hat{\theta}_{m^{\ast},n}(\lambda))-f_{m}(\theta_{m^{\ast},0})\big) \\
&=(\hat{\theta}_{m,n}(\lambda)-\theta_{m,0})-F_{m}(\hat{\theta}_{m^{\ast},n}(\lambda)-\theta_{m^{\ast},0}) \\
&=O_{p}(n^{-1/2}), \\
\sup_{\theta_{m}\in\Theta_{m}}\left|\frac{1}{n}\p_{\theta_{m}}^{2}\mbbh_{m,n}(\theta_{m};\lambda)\right|&=O_{p}(1).
\end{align*}
Therefore, 
\begin{align*}
&P\left[\betabic^{(m)}_{n}>\betabic^{(m^{\ast}_{m})}_{n}\right] \\
&=P\bigg[\left(-\frac{1}{n}\p_{\theta_{m}}^{2}\mbbh_{m,n}(\tilde{\theta}_{m,n}(\lambda);\lambda)\right)\left[\left(\sqrt{n}\big(\hat{\theta}_{m,n}(\lambda)-f_{m}(\hat{\theta}_{m^{\ast},n}(\lambda))\big)\right)^{\otimes2}\right] \\
&\hspace{20mm} < (p_{m}-p_{m^{\ast}})\log n\bigg] \\
&\to1
\end{align*}
as $n\to\infty$.

\subsubsection{Proof of (ii)}
Recall that for any $m\in\{1,\ldots,M\}$,
\begin{align*}
\mathbb{Y}_{m,0}(\theta_{m})
&=-\frac{1}{2T}\int_{0}^{T}\bigg\{\log\left(\frac{\det S_{m}(X_{t},\theta_{m})}{\det S(X_{t})}\right) \\
&\hspace{25mm}+\trace\left(S_{m}(X_{t},\theta_{m})^{-1}S(X_{t})-I_{d}\right)\bigg\}dt,
\end{align*}
and for any $m_{t}\in\mathfrak{M}$,
\begin{align*}
\mathbb{Y}_{m_{t},n}(\theta_{m_{t}};\lambda)
&=\frac{1}{n}\big(\mbbh_{m_{t},n}(\theta_{m_{t}};\lambda)-\mbbh_{m_{t},n}(\theta_{m_{t},0};\lambda)\big), \\
\mathbb{Y}_{m_{t},n}^{\flat}(\theta_{m_{t}};\lambda)
&=\frac{1}{n}\big(\mbbh_{m_{t},n}^{\flat}(\theta_{m_{t}};\lambda)-\mbbh_{m_{t},n}^{\flat}(\theta_{m_{t},0};\lambda)\big).
\end{align*}
For any $m_{c}\in\mathfrak{M}^{c}$, we define 
\begin{align*}
\check{\mathbb{Y}}_{m_{c},n}(\theta_{m_{c}};\lambda)
&=\frac{1}{n}\big(\mbbh_{m_{c},n}(\theta_{m_{c}};\lambda)-\mbbh_{m^{\ast},n}(\theta_{m^{\ast},0};\lambda)\big), \\
\check{\mathbb{Y}}_{m_{c},n}^{\flat}(\theta_{m_{c}};\lambda)
&=\frac{1}{n}\big(\mbbh_{m_{c},n}^{\flat}(\theta_{m_{c}};\lambda)-\mbbh_{m^{\ast},n}^{\flat}(\theta_{m^{\ast},0};\lambda)\big).
\end{align*}
Furthermore, under Assumption \ref{se:ass_modconsis} (ii), for any $m_{c}\in\mathfrak{M}^{c}$, we define
\begin{align*}
\bar{\mathbb{Y}}_{m_{c},n}(\theta_{m_{c}};\lambda)
&=\frac{1}{n}\big(\mbbh_{m_{c},n}(\theta_{m_{c}};\lambda)-\mbbh_{m_{c},n}(\bar{\theta}_{m_{c}};\lambda)\big), \\
\bar{\mathbb{Y}}_{m_{c},n}^{\flat}(\theta_{m_{c}};\lambda)
&=\frac{1}{n}\big(\mbbh_{m_{c},n}^{\flat}(\theta_{m_{c}};\lambda)-\mbbh_{m_{c},n}^{\flat}(\bar{\theta}_{m_{c}};\lambda)\big).
\end{align*}

We will apply the following lemmas for verifying \eqref{se:model_consis3} and \eqref{se:model_consis4}.

\begin{lem}
\label{se:lem1}
Suppose that Assumptions \ref{hm:A_diff.coeff}--\ref{hm:A_lam} hold.
Then, for any $m_{t}\in\mathfrak{M}$, we have
\begin{align*}
\mathbb{Y}_{m_{t},n}(\theta_{m_{t}};\lambda)&\cip\mathbb{Y}_{m_{t},0}(\theta_{m_{t}}), \\
\mathbb{Y}_{m_{t},n}^{\flat}(\theta_{m_{t}};\lambda)&\cip\mathbb{Y}_{m_{t},0}(\theta_{m_{t}})
\end{align*}
uniformly in $\theta_{m_{t}}$.
Moreover, for any $m_{c}\in\mathfrak{M}^{c}$, we have
\begin{align*}
\check{\mathbb{Y}}_{m_{c},n}(\theta_{m_{c}};\lambda)&\cip\mathbb{Y}_{m_{c},0}(\theta_{m_{c}}), \\
\check{\mathbb{Y}}_{m_{c},n}^{\flat}(\theta_{m_{c}};\lambda)&\cip\mathbb{Y}_{m_{c},0}(\theta_{m_{c}})
\end{align*}
uniformly in $\theta_{m_{c}}$.
\end{lem}
\begin{proof}
Let
\begin{align*}
V_{m}(x,\theta_{m}):=S(x)^{-1/2}S_{m}(x,\theta_{m})S(x)^{-1/2}.
\end{align*}
for any $m\in\{1,\ldots,M\}$.
When $m_{t}\in\mathfrak{M}$, $V_{m_{t}}(x,\theta_{m_{t},0})=I_{d}$.
From \cite[Appendix B]{EguMas25}, for any $m_{t}\in\mathfrak{M}$
and fixed $\lambda>0$, it holds that
\begin{align}
&\mathbb{Y}_{m_{t},n}(\theta_{m_{t}};\lambda) \\
&\cip\frac{(2\pi)^{-\lambda d/2}}{T}\int_{0}^{T}\bigg[\det\big(S_{m_{t}}(X_{t},\theta_{m_{t}})\big)^{-\lambda/2}\bigg\{\frac{1}{\lambda}\det\big(\lambda V_{m_{t}}(X_{t},\theta_{m_{t}})^{-1}+I_{d}\big)^{-1/2} \\
&\hspace{65mm} -\frac{1}{(\lambda+1)^{1+d/2}}\bigg\} \\
&\hspace{25mm} -\det\big(S(X_{t})\big)^{-\lambda/2}\left(\frac{1}{\lambda}\det\big((\lambda+1)I_{d}\big)^{-1/2}-\frac{1}{(\lambda+1)^{1+d/2}}\right)\bigg]dt \\
&=\frac{(2\pi)^{-\lambda d/2}}{T}\int_{0}^{T}\bigg[\frac{1}{(\lambda+1)^{1+d/2}}\left(\det\big(S(X_{t})\big)^{-\lambda/2}-\det\big(S_{m_{t}}(X_{t},\theta_{m_{t}})\big)^{-\lambda/2}\right) \\
&\hspace{28mm} +\frac{1}{\lambda}\bigg\{\det\big(S_{m_{t}}(X_{t},\theta_{m_{t}})\big)^{-\lambda/2}\det\big(\lambda V_{m_{t}}(X_{t},\theta_{m_{t}})^{-1}+I_{d}\big)^{-1/2} \\
&\hspace{35mm} -\det\big(S(X_{t})\big)^{-\lambda/2}\det\big((\lambda+1)I_{d}\big)^{-1/2}\bigg\}\bigg]dt \\
&=:\mathbb{Y}_{m_{t},0}(\theta_{m_{t}};\lambda), \label{se:unifc_y_beta}\\
&\mathbb{Y}_{m_{t},n}^{\flat}(\theta_{m_{t}};\lambda) \\
&\cip\frac{(2\pi)^{-\lambda d/2}}{T}\int_{0}^{T}\frac{1}{\lambda}\bigg\{\det\big(S_{m_{t}}(X_{t},\theta_{m_{t}})\big)^{-\lambda/2(\lambda+1)}\det\big(\lambda V_{m_{t}}(X_{t},\theta_{m_{t}})^{-1}+I_{d}\big)^{-1/2} \\
&\hspace{32mm} -\det\big(S(X_{t})\big)^{-\lambda/2(\lambda+1)}\det\big((\lambda+1)I_{d}\big)^{-1/2}\bigg\}dt \\
&=:\mathbb{Y}_{m_{t},0}^{\flat}(\theta_{m_{t}};\lambda) \label{se:unifc_y_hol}
\end{align}
uniformly in $\theta_{m_{t}}$.
Incorporating Assumption \ref{hm:A_lam} into this framework, 
we have
\begin{align*}
\mathbb{Y}_{m_{t},n}(\theta_{m_{t}};\lambda)&\cip\lim_{\lambda\to0}\mathbb{Y}_{m_{t},0}(\theta_{m_{t}};\lambda), \\
\mathbb{Y}_{m_{t},n}^{\flat}(\theta_{m_{t}};\lambda)&\cip\lim_{\lambda\to0}\mathbb{Y}_{m_{t},0}^{\flat}(\theta_{m_{t}};\lambda)
\end{align*}
uniformly in $\theta_{m_{t}}$, analogously to \eqref{se:unifc_y_beta} and \eqref{se:unifc_y_hol}.
Thus, it suffices to consider the limits $\lim_{\lambda\to0}\mathbb{Y}_{m_{t},0}(\theta_{m_{t}};\lambda)$ and $\lim_{\lambda\to0}\mathbb{Y}_{m_{t},0}^{\flat}(\theta_{m_{t}};\lambda)$.
By Taylor expansions in $\lambda$ around zero, we have
\begin{align*}
\det\big(S_{m_{t}}(X_{t},\theta_{m_{t}})\big)^{-\lambda/2}
&=1-\frac{1}{2}\log\big(\det\big(S_{m_{t}}(X_{t},\theta_{m_{t}})\big)\big)\lambda+O_{p}(\lambda^{2}), \\
\det\big(S(X_{t})\big)^{-\lambda/2}
&=1-\frac{1}{2}\log\big(\det\big(S(X_{t})\big)\big)\lambda+O_{p}(\lambda^{2}), \\
\det\big(S_{m_{t}}(X_{t},\theta_{m_{t}})\big)^{-\lambda/2(\lambda+1)}
&=1-\frac{1}{2}\log\big(\det\big(S_{m_{t}}(X_{t},\theta_{m_{t}})\big)\big)\lambda+O_{p}(\lambda^{2}), \\
\det\big(S(X_{t})\big)^{-\lambda/2(\lambda+1)}
&=1-\frac{1}{2}\log\big(\det\big(S(X_{t})\big)\big)\lambda+O_{p}(\lambda^{2}), \\
\det\big(\lambda V_{m_{t}}(X_{t},\theta_{m_{t}})^{-1}+I_{d}\big)^{-1/2}
&=1-\frac{1}{2}\trace\big(V_{m_{t}}(X_{t},\theta_{m_{t}})^{-1}\big)\lambda+O_{p}(\lambda^{2}) \\
&=1-\frac{1}{2}\trace\big(S_{m_{t}}(X_{t},\theta_{m_{t}})^{-1}S(X_{t})\big)\lambda+O_{p}(\lambda^{2}), \\
\det\big((\lambda+1)I_{d}\big)^{-1/2} 
&=1-\frac{1}{2}\trace(I_{d})\lambda+O_{p}(\lambda^{2}),
\end{align*}
and the $O_{p}(\lambda^{2})$ terms are valid uniformly in $\theta_{m_{t}}$; for example,
\begin{equation}
    \sup_{\theta_{m_{t}}\in\Theta_{m_{t}}}\left|
    \det\big(S_{m_{t}}(X_{t},\theta_{m_{t}})\big)^{-\lambda/2}
    - \left( 1-\frac{1}{2}\log\big(\det\big(S_{m_{t}}(X_{t},\theta_{m_{t}})\big)\big)\lambda \right)
    \right|
    =O_{p}(\lambda^{2}).
\end{equation}
Hence, we obtain
\begin{align*}
&\lim_{\lambda\to0}\mathbb{Y}_{m_{t},0}(\theta_{m_{t}};\lambda) \\
&=\lim_{\lambda\to0}\frac{(2\pi)^{-\lambda d/2}}{T}\int_{0}^{T}\bigg[O_{p}(\lambda)+\frac{1}{\lambda}\bigg\{\Big(1-\frac{1}{2}\log\big(\det\big(S_{m_{t}}(X_{t},\theta_{m_{t}})\big)\big)\lambda \\
&\hspace{57mm} -\frac{1}{2}\trace\big(S_{m_{t}}(X_{t},\theta_{m_{t}})^{-1}S(X_{t})\big)\lambda+O_{p}(\lambda^{2})\Big) \\
&\hspace{50mm} -\Big(1-\frac{1}{2}\log\big(\det\big(S(X_{t})\big)\big)\lambda-\frac{1}{2}\trace(I_{d})\lambda \\
&\hspace{57mm} +O_{p}(\lambda^{2})\Big)\bigg\}\bigg]dt \\
&=\lim_{\lambda\to0}-\frac{(2\pi)^{-\lambda d/2}}{2T}\int_{0}^{T}\left\{\log\bigg(\frac{\det \big(S_{m_{t}}(X_{t},\theta_{m_{t}})\big)}{\det \big(S(X_{t})\big)}\right) \\
&\hspace{37mm} +\trace\left(S_{m_{t}}(X_{t},\theta_{m_{t}})^{-1}S(X_{t})-I_{d}\right)+O_{p}(\lambda)\bigg\}dt \\
&=\mathbb{Y}_{m_{t},0}(\theta_{m_{t}}), \\
&\lim_{\lambda\to0}\mathbb{Y}_{m_{t},0}^{\flat}(\theta_{m_{t}};\lambda) \\
&=\lim_{\lambda\to0}\frac{(2\pi)^{-\lambda d/2}}{T}\int_{0}^{T}\bigg[\frac{1}{\lambda}\bigg\{\Big(1-\frac{1}{2}\log\big(\det\big(S_{m_{t}}(X_{t},\theta_{m_{t}})\big)\big)\lambda \\
&\hspace{41mm} -\frac{1}{2}\trace\big(S_{m_{t}}(X_{t},\theta_{m_{t}})^{-1}S(X_{t})\big)\lambda+O_{p}(\lambda^{2})\Big) \\
&\hspace{34mm} -\Big(1-\frac{1}{2}\log\big(\det\big(S(X_{t})\big)\big)\lambda-\frac{1}{2}\trace(I_{d})\lambda+O_{p}(\lambda^{2})\Big)\bigg\}\bigg]dt \\
&=\mathbb{Y}_{m_{t},0}(\theta_{m_{t}}).
\end{align*}
uniformly in $\theta_{m_{t}}$.
Therefore, 
\begin{align*}
\mathbb{Y}_{m_{t},n}(\theta_{m_{t}};\lambda)\cip\mathbb{Y}_{m_{t},0}(\theta_{m_{t}}), \quad \mathbb{Y}_{m_{t},n}^{\flat}(\theta_{m_{t}};\lambda)\cip\mathbb{Y}_{m_{t},0}(\theta_{m_{t}})
\end{align*}
uniformly in $\theta_{m_{t}}$.

Similarly, for any $m_{c}\in\mathfrak{M}^{c}$ and fixed $\lambda>0$, we have
\begin{align*}
\check{\mathbb{Y}}_{m_{c},n}(\theta_{m_{c}};\lambda)
\cip\mathbb{Y}_{m_{c},0}(\theta_{m_{c}};\lambda), 
\quad \check{\mathbb{Y}}_{m_{c},n}^{\flat}(\theta_{m_{c}};\lambda)
\cip\mathbb{Y}_{m_{c},0}^{\flat}(\theta_{m_{c}};\lambda)
\end{align*}
uniformly in $\theta_{m_{c}}$.
Since we have
\begin{align*}
\lim_{\lambda\to0}\mathbb{Y}_{m_{c},0}(\theta_{m_{c}};\lambda)=\mathbb{Y}_{m_{c},0}(\theta_{m_{c}}), 
\quad \lim_{\lambda\to0}\mathbb{Y}_{m_{c},0}^{\flat}(\theta_{m_{c}};\lambda)=\mathbb{Y}_{m_{c},0}(\theta_{m_{c}})
\end{align*}
uniformly in $\theta_{m_{c}}$, it follows that
\begin{align*}
\check{\mathbb{Y}}_{m_{c},n}(\theta_{m_{c}};\lambda)&\cip\mathbb{Y}_{m_{c},0}(\theta_{m_{c}}), \quad \check{\mathbb{Y}}_{m_{c},n}^{\flat}(\theta_{m_{c}};\lambda)\cip\mathbb{Y}_{m_{c},0}(\theta_{m_{c}}).
\end{align*}
\end{proof}

For any $m_{t}\in\mathfrak{M}$, since $\hat{\theta}_{m_{t},n}(\lambda)\cip\theta_{m_{t},0}$ and $\hat{\theta}_{m_{t},n}^{\flat}(\lambda)\cip\theta_{m_{t},0}$ (see \cite[Theorem 3.4]{EguMas25}), it follows from Lemma \ref{se:lem1} that $\mathbb{Y}_{m_{t},n}(\hat{\theta}_{m_{t},n}(\lambda);\lambda)\cip0$ and $\mathbb{Y}_{m_{t},n}^{\flat}(\hat{\theta}_{m_{t},n}^{\flat}(\lambda);\lambda)\cip0$.

\begin{lem}
\label{se:lem2}
Suppose that Assumptions \ref{hm:A_diff.coeff}--\ref{hm:A_lam}, and \ref{se:ass_modconsis} (ii) hold.
Then, for any $m_{c}\in\mathfrak{M}^{c}$, we have
\begin{align*}
\bar{\mathbb{Y}}_{m_{c},n}(\hat{\theta}_{m_{c},n}(\lambda);\lambda)
&\cip0, \\
\bar{\mathbb{Y}}_{m_{c},n}^{\flat}(\hat{\theta}_{m_{c},n}^{\flat}(\lambda);\lambda)
&\cip0.
\end{align*}
\end{lem}
\begin{proof}
In a similar way as Lemma \ref{se:lem1}, for any $m_{c}\in\mathfrak{M}^{c}$, we obtain
\begin{align*}
\bar{\mathbb{Y}}_{m_{c},n}(\theta_{m_{c}};\lambda)
&\cip-\frac{1}{2T}\int_{0}^{T}\left\{\log\bigg(\frac{\det \big(S_{m_{c}}(X_{t},\theta_{m_{c}})\big)}{\det \big(S_{m_{c}}(X_{t},\bar{\theta}_{m_{c}})\big)}\right) \\
&\quad+\trace\left(S_{m_{c}}(X_{t},\theta_{m_{c}})^{-1}S(X_{t})-S_{m_{c}}(X_{t},\bar{\theta}_{m_{c}})^{-1}S(X_{t})\right)\bigg\}dt, \\
&=:\bar{\mathbb{Y}}_{m_{c},0}(\theta_{m_{c}}), \\
\bar{\mathbb{Y}}_{m_{c},n}^{\flat}(\theta_{m_{c}};\lambda)
&\cip\bar{\mathbb{Y}}_{m_{c},0}(\theta_{m_{c}})
\end{align*}
uniformly in $\theta_{m_{c}}$.
Since $\hat{\theta}_{m_{c},n}(\lambda)\cip\bar{\theta}_{m_{c}}$ and $\hat{\theta}_{m_{c},n}^{\flat}(\lambda)\cip\bar{\theta}_{m_{c}}$,
\begin{align*}
\bar{\mathbb{Y}}_{m_{c},n}(\hat{\theta}_{m_{c},n}(\lambda);\lambda)
&\cip\bar{\mathbb{Y}}_{m_{c},0}(\bar{\theta}_{m_{c}})=0, \\
\bar{\mathbb{Y}}_{m_{c},n}^{\flat}(\hat{\theta}_{m_{c},n}^{\flat}(\lambda);\lambda)
&\cip\bar{\mathbb{Y}}_{m_{c},0}(\bar{\theta}_{m_{c}})=0.
\end{align*}
\end{proof}

Applying Lemmas \ref{se:lem1} and \ref{se:lem2}, we prove \eqref{se:model_consis3}.
To this end, it is sufficient to show that 
\begin{align*}
P\left[\min_{m_{c}\in\mathfrak{M}^{c}}\betabic^{(m_{c})}_{n}\leq\betabic^{(m^{\ast})}_{n}\right]\to0
\end{align*}
as $n\to\infty$.
We obtain
\begin{align*}
0&\leq P\left[\min_{m_{c}\in\mathfrak{M}^{c}}\betabic^{(m_{c})}_{n}\leq\betabic^{(m^{\ast})}_{n}\right] \\
&=P\left[\bigcup_{m_{c}\in\mathfrak{M}^{c}}\left(\betabic^{(m)}_{n}\leq\betabic^{(m^{\ast})}_{n}\right)\right] \\
&\leq \sum_{m_{c}\in\mathfrak{M}^{c}} P\left[\betabic^{(m_{c})}_{n}\leq\betabic^{(m^{\ast})}_{n}\right] \\
&=\sum_{m_{c}\in\mathfrak{M}^{c}} P\left[\frac{1}{n}\left(\betabic^{(m_{c})}_{n}-\betabic^{(m^{\ast})}_{n}\right)\leq0\right].
\end{align*}

\begin{itemize}
\item When Assumption \ref{se:ass_modconsis} (i) holds, for any $m_{c}\in\mathfrak{M}^{c}$, it follows from Lemma \ref{se:lem1} that
\begin{align}
&P\left[\frac{1}{n}\left(\betabic^{(m_{c})}_{n}-\betabic^{(m^{\ast})}_{n}\right)\leq0\right] \nn\\
&=P\bigg[-\frac{2}{n}\left(\mbbh_{m_{c},n}(\hat{\theta}_{m_{c},n}(\lambda);\lambda)-\mbbh_{m^{\ast},n}(\hat{\theta}_{m^{\ast},n}(\lambda);\lambda)\right) \nn\\
&\hspace{10mm} -(p_{m^{\ast}}-p_{m_{c}})\frac{\log n}{n}\leq0\bigg] \nn\\
&=P\bigg[\frac{2}{n}\left(\mbbh_{m_{c},n}(\hat{\theta}_{m_{c},n}(\lambda);\lambda)-\mbbh_{m^{\ast},n}(\theta_{m^{\ast},0};\lambda)\right) \\
&\hspace{10mm} -\frac{2}{n}\left(\mbbh_{m^{\ast},n}(\hat{\theta}_{m^{\ast},n}(\lambda);\lambda)-\mbbh_{m^{\ast},n}(\theta_{m^{\ast},0};\lambda)\right) \nn\\
&\hspace{10mm} -(p_{m^{\ast}}-p_{m_{c}})\frac{\log n}{n}\geq0\bigg] \nn\\
&\leq P\left[\sup_{\theta_{m_{c}}\in\bar{\Theta}_{m_{c}}}2\check{\mathbb{Y}}_{m_{c},n}(\theta_{m_{c}};\lambda)-2\mathbb{Y}_{m^{\ast},n}(\hat{\theta}_{m^{\ast},n}(\lambda);\lambda)-(p_{m^{\ast}}-p_{m_{c}})\frac{\log n}{n}\geq0\right] \nn\\
&=P\left[\sup_{\theta_{m_{c}}\in\bar{\Theta}_{m_{c}}}\mathbb{Y}_{m_{c},0}(\theta_{m_{c}})+o_{p}(1)\geq0\right] \nn\\
&\to0. \label{se:misconsis1}
\end{align}

\item When Assumption \ref{se:ass_modconsis} (ii) holds, for any $m_{c}\in\mathfrak{M}^{c}$, it follows from Lemmas \ref{se:lem1} and \ref{se:lem2} that
\begin{align}
&P\left[\frac{1}{n}\left(\betabic^{(m_{c})}_{n}-\betabic^{(m^{\ast})}_{n}\right)\leq0\right] \nn\\
&=P\bigg[-\frac{2}{n}\left(\mbbh_{m_{c},n}(\hat{\theta}_{m_{c},n}(\lambda);\lambda)-\mbbh_{m^{\ast},n}(\hat{\theta}_{m^{\ast},n}(\lambda);\lambda)\right) \nn\\
&\hspace{10mm} -(p_{m^{\ast}}-p_{m_{c}})\frac{\log n}{n}\leq0\bigg] \nn\\
&=P\bigg[\frac{2}{n}\left(\mbbh_{m_{c},n}(\hat{\theta}_{m_{c},n}(\lambda);\lambda)-\mbbh_{m_{c},n}(\bar{\theta}_{m_{c}};\lambda)\right) \\
&\hspace{10mm} -\frac{2}{n}\left(\mbbh_{m^{\ast},n}(\hat{\theta}_{m^{\ast},n}(\lambda);\lambda)-\mbbh_{m^{\ast},n}(\theta_{m^{\ast},0};\lambda)\right) \nn\\
&\hspace{10mm} +\frac{2}{n}\left(\mbbh_{m_{c},n}(\bar{\theta}_{m_{c}};\lambda)-\mbbh_{m^{\ast},n}(\theta_{m^{\ast},0};\lambda)\right) \nn\\
&\hspace{10mm} -(p_{m^{\ast}}-p_{m_{c}})\frac{\log n}{n}\geq0\bigg] \nn\\
&\leq P\bigg[2\bar{\mathbb{Y}}_{m_{c},n}(\hat{\theta}_{m_{c},n}(\lambda);\lambda)-2\mathbb{Y}_{m^{\ast},n}(\hat{\theta}_{m^{\ast},n}(\lambda);\lambda)+2\check{\mathbb{Y}}_{m_{c},n}(\bar{\theta}_{m_{c}};\lambda) \nn\\
&\hspace{10mm} -(p_{m^{\ast}}-p_{m_{c}})\frac{\log n}{n}\geq0\bigg] \nn\\
&=P\left[\mathbb{Y}_{m_{c},0}(\bar{\theta}_{m_{c}})+o_{p}(1)\geq0\right] \nn\\
&\to0. \label{se:misconsis2}
\end{align}
\end{itemize}

From \eqref{se:misconsis1} and \eqref{se:misconsis2}, we have
\begin{align*}
\sum_{m_{c}\in\mathfrak{M}^{c}} P\left[\frac{1}{n}\left(\betabic^{(m_{c})}_{n}-\betabic^{(m^{\ast})}_{n}\right)\leq0\right]
&\to0.
\end{align*}
Therefore,
\begin{align*}
P\left[\min_{m_{c}\in\mathfrak{M}^{c}}\betabic^{(m_{c})}_{n}\leq\betabic^{(m^{\ast})}_{n}\right]&\to0
\end{align*}
as $n\to\infty$, and \eqref{se:model_consis3} is established.
In an analogous way, \eqref{se:model_consis4} can be shown.

\bigskip

\noindent
\textbf{Acknowledgements.}
The authors are grateful to the anonymous reviewers and editors for their valuable comments, which led to improvements in the paper.
This work was partially supported by JST CREST Grant Number JPMJCR2115 and JSPS KAKENHI Grant Numbers 
JP23K22410 (HM), and JP24K16971 (SE), Japan.

\bigskip
On behalf of all authors, the corresponding author states that there is no conflict of interest.
\bigskip 
\bibliographystyle{abbrv} 

\appendix
\section{Auxiliary results} \label{sec:appe}

\subsection{Robustified Gaussian quasi-likelihood inference}
In this section, we briefly present the robustified Gaussian quasi-likelihood inference following \cite{EguMas25}.

We recall that $\pr_\theta$ denotes the distribution of the random elements
\begin{equation}
\left(Y,X,\mu,\mu',\sig',w,w',J,J'\right)
\nonumber
\end{equation}
associated with $\theta\in\overline{\Theta}$.
We denote by $\E_{\theta}$ the corresponding expectation.
The conditional $\pr_\theta$-probability given $\mcf_{t_{j-1}}$ and the associated conditional expectation are denoted by $\pr^{j-1}_\theta[\cdot]:=\pr_\theta[\cdot|\mcf_{t_{j-1}}]$ and $\E^{j-1}_\theta[\cdot]:=\E_\theta[\cdot|\mcf_{t_{j-1}}]$, respectively.
We write $P=P_{\tz}$, $E=E_{\tz}$, $\pr^{j-1}[\cdot]=\pr^{j-1}_{\tz}[\cdot]$, and $\E^{j-1}[\cdot]=\E^{j-1}_{\tz}[\cdot]$.

Let $N_t := \sum_{0<s\le t} I(\D Y_s \ne 0)$, and let $N'_t := \sum_{0<s\le t} I(\D X_s \ne 0)$, where $\D Y_{s} := Y_{s} - Y_{s-}$ and $\D X_{s} := X_{s} - X_{s-}$ denote the jump size of $Y$ and $X$ at time $s$, respectively.
We use the shorthands $\sup_\theta$ and $\inf_\theta$ for $\sup_{\theta\in\overline{\Theta}}$ and $\inf_{\theta\in\overline{\Theta}}$, respectively.
Moreover, we write $C>0$ for a universal positive constant which may vary at each appearance, and $a_{n} \lesssim b_{n}$ for possibly random nonnegative sequences $(a_n)$ and $(b_n)$ if $a_{n}\le C b_{n}$ a.s. holds for every $n$ large enough.
The following assumptions are imposed to derive the asymptotic properties of density-power and the H\"{o}lder-based GQMLEs.

\begin{ass}
\label{hm:A_diff.coeff}~
\begin{enumerate}
\item The function $(x,\theta) \mapsto S(x,\theta)$ belongs to the class $\mcc^{2,4}(\mbbr^{d'} \times \Theta)$, with all the partial derivatives continuous in $\overline{\Theta}$ for each $x$.

\item 
$\ds{\sup_\theta |\p_x^k \p_\theta^l S(x,\theta)| \lesssim (1+|x|)^C}$ and there exists a constant $c_S'\ge 0$ such that
\begin{equation*}
\inf_{\theta}\lam_{\min}(S(x,\theta)) \gtrsim (1+|x|)^{-c_S'}
\nn
\end{equation*}
for $x\in\mbbr^{d'}$, where $\lam_{\min}(S(x,\theta))$ denotes the minimum eigenvalue of $S(x,\theta)$.
\end{enumerate}
\end{ass}

\begin{ass}
\label{hm:A_J}~
\begin{enumerate}
\item The numbers of jumps of $Y$ and $X$ are almost surely finite in $[0,T]$:
\begin{equation*}
    \pr\left[\max\{N_T,\, N'_T\} <\infty\right] =1.
\end{equation*}
\item There exist constants $\kappa>1/2$ and $c_1 \ge 0$ for which
\begin{equation*}\label{hm:A_J-1}
\pr^{j-1}\left[ \D_j N + \D_j N' \ge 1\right] \lesssim (1+|X_{t_{j-1}}|^{c_1}) \, h^{\kappa}
\end{equation*}
for $j=1,\dots,n$.
\item $\ds{\sup_{t\le T}\E[|J_t'|^K]<\infty}$ for any $K>0$, and
\begin{equation*}\label{hm:A_J-2}
    \sup_{t,s\in[0,T];\atop |t-s|\le h} \E\left[|J'_t - J'_s|^2\right] \lesssim h^{c'}
\end{equation*}
for some $c'>0$.
\end{enumerate}
\end{ass}

Assumptions \ref{hm:A_diff.coeff} and \ref{hm:A_J} concern the diffusion coefficient and the jump structure, respectively.
Under Assumption \ref{hm:A_J}, we have $J_{t} = \sum_{0<s\le t} \D Y_{s}$ and $J_{t}^{\prime} = \sum_{0<s\le t} \D X_{s}$.

\begin{ass}
\label{hm:A_drif.coeff&X}~
$\mu=(\mu_t)_{t\le T}$, $\mu'=(\mu'_t)_{t\le T}$ and $\sig'=(\sig'_t)_{t\le T}$ are $(\mcf_t)$-adapted {\cadlag} processes in $\mbbr^d$, $\mbbr^{d'}$, $\mbbr^{d'}\otimes\mbbr^{r'}$, respectively, such that 
\begin{align*}
& \sup_{t\in[0,T]} \E\left[|X_0|^K + |\mu_t|^K + |\mu'_t|^K + |\sig'_t|^K\right] <\infty,
\nn\\
& 
\max_{1\le j\le n}\sup_{t\in(t_{j-1},t_j)} 
\left(
\E\left[|\mu_t - \mu_{t_{j-1}}|^K; G_j\right]
+ 
\E\left[ |\sig'_t - \sig'_{t_{j-1}}|^2 \right] 
\right)
=o(1)
\end{align*}
for any $K\ge 2$, where $G_{j} = \left\{ \D_j N = 0,~ \D_j N' = 0\right\}$.
\end{ass}

\begin{ass}
\label{hm:A_iden}
We have
\begin{equation*}
\pr\big[\forall t\in[0,T],~S(X_t,\theta) = S(X_t,\tz)\big]=1
\label{hm:assump_iden}
\end{equation*}
if and only if $\theta=\tz$.
\end{ass}

\begin{ass}
\label{hm:A_lam}~
We have $\lam_n\to 0$ in such a way that for $\kappa>1/2$ in Assumption \ref{hm:A_J} \eqref{hm:A_J-1},
\begin{equation*}
\label{hm:lam.condition-2}
\frac{\sqrt{n}\,h^{\kappa}}{\lam_n} \to 0.
\end{equation*}
\end{ass}



\subsection{Basic tool for deriving the BIC}
For convenience, we provide a set of general conditions under which a quasi-marginal log-likelihood admits a Schwarz-type stochastic expansion.

\medskip

Given an underlying probability space $(\Omega,\mcf,\pr)$, let $\mbbh_{n}: \Theta\times \Omega\to\mbbr$ be a $\mcc^{3}(\Theta)$-random function, where $\Theta\subset\mbbr^{p}$ is a bounded convex domain, and let $\tz\in\Theta$ be a constant.
Let 
\begin{equation}
\D_n=\Delta_{n}(\tz):=n^{-1/2}\p_{\theta}\mbbh_{n}(\tz), \quad \Gam_n=\Gam_n(\tz):=-n^{-1}\p_{\theta}^{2}\mbbh_{n}(\tz). \nn
\end{equation}
We define the random field on $\mbbr^{p}$ associated with $\mbbh_{n}$ by
\begin{equation}
\mbbz_{n}(u):=\exp\left( \mbbh_{n}(\tz+n^{-1/2}u) - \mbbh_{n}(\tz) \right),
\nn
\end{equation}
and $\mbbz_{n}\equiv 0$ outside the set $\mbbu_{n}=\{u\in\mbbr^{p};\tz+n^{-1/2}u\in\Theta\}\subset\mbbr^{p}$.
We also define $\mbby_{n}(\theta):=n^{-1}\left(\mbbh_{n}(\theta)-\mbbh_{n}(\tz)\right)$ and let $\mbby(\theta)$ be an $\mcf$-measurable $\mbbr$-valued random function.
We consider a bounded prior density $\pi(\theta)$ on $\Theta$, which is continuous and positive at $\tz$.

\begin{thm}
\label{hm:thm.qbic}
In addition to the above setting, we assume the following conditions.
\begin{itemize}
\item Let $\Sig_0$ and $\Gam_0$ be almost surely positive definite random matrices in $\mbbr^{p\times p}$. 
The following joint convergence in distribution holds:
\begin{equation}
\left(\Delta_{n},\, \Gam_n\right) \overset{\mcl}\to \big( \Sig_{0}^{1/2}\eta,\, \Gam_{0}\big),
\label{appendix.thm1}
\end{equation}
where $\eta \sim N_{p}(0,I_{p})$ independent of $\mcf$, defined on an extension of the underlying probability space.

\item We have
\begin{align}
\sup_{\theta}\left| \frac1n \p_{\theta}^{3}\mbbh_{n}(\theta) \right| = O_{p}(1). \label{appendix.thm4}
\end{align}
\noeqref{appendix.thm4}

\item There exists a constant $\ep\in(0,1/2]$ such that
\begin{equation}
\sup_{\theta}\big| n^{\ep}\left(\mbby_{n}(\theta)-\mbby(\theta)\right)\big| = O_p(1).
\label{appendix.thm5}
\end{equation}
\noeqref{appendix.thm5}

\item There exists an $\mcf$-measurable random variable $\chi_{0}$ that is almost surely positive such that, for each $\kappa>0$,
\begin{equation}
\sup_{\theta:\, |\theta-\tz|\ge\kappa}\mbby(\theta) \le -\chi_{0}\kappa^{2}\qquad \text{a.s.}
\label{appendix.thm6}
\end{equation}

\end{itemize}
Then, for any $\tes \in \argmax_{\theta}\mbbh_{n}(\theta)$, we have 
\begin{equation}
\sqrt{n}(\tes-\tz) = \Gam_{0}^{-1}\Delta_{n} + o_p(1) \cil \Gam_{0}^{-1}\Sigma_{0}^{1/2}\eta \sim MN_p(0,\Gam_0^{-1}\Sig_0 \Gam_0^{-1})
\nonumber
\end{equation}
and
\begin{align}
\int \left|\mbbz_{n}(u)\,\pi(\tz+n^{-1/2}u) - \mbbz_{n}^{0}(u)\,\pi(\tz) \right|du \cip 0, \label{se:thm.qbic.cip}
\end{align}
where
\begin{equation}
\mbbz^{0}_{n}(u) =\exp\bigg( \Delta_{n}[u] - \frac{1}{2}\Gam_{0}[u,u] \bigg).
\nonumber
\end{equation}
\end{thm}

Theorem \ref{hm:thm.qbic} can be deduced as a single-parameter version of \cite[Proof of Theorem 1]{JasKamMas19}. 
The change of variable $\theta=\tz+n^{-1/2}u$ and \eqref{se:thm.qbic.cip} implies the stochastic expansion
\begin{align}
&-\log\bigg(\int_{\Theta}\exp\{\mbbh_{n}(\theta)\}\pi(\theta)d\theta\bigg)
\nn \\
&=-\mbbh_{n}(\tz) + \frac{p}{2}\log n-\log\left\{\int_{\mbbu_{n}}\mbbz_{n}(u)\,\pi(\tz+n^{-1/2}u)du\right\} \nn \\
&=-\mbbh_{n}(\tz) + \frac{p}{2}\log n - \log\pi(\tz) - \frac{p}{2}\log(2\pi)
+\frac{1}{2}\log|\Gam_{0}| - \frac{1}{2}\Gam_{0}^{-1}\left[\D_{n}^{\otimes 2}\right] + o_{p}(1) \nn \\
&=-\mbbh_{n}(\tz) + \frac{p}{2}\log n+O_{p}(1). \label{se:thm.qbic.cip2}
\end{align}

In the proof of \cite[Theorem 3.4]{EguMas25}, it is shown that \eqref{appendix.thm1}--\eqref{appendix.thm6} hold when the density-power and the H\"{o}lder-based GQLFs are taken as the random function $\mbbh_{n}$.
Hence, Theorem \ref{hm:thm.qbic} holds in our model setting under Assumptions \ref{hm:A_diff.coeff}--\ref{hm:A_lam}, and \ref{se:prior}.


\section{Additional numerical experiment} \label{se:app2}
In this section, we present a numerical experiment in the absence of jumps. 
The setup of the numerical experiment is the same as that in Section \ref{se:simu1}, except for the data-generating model.
We consider the following model, obtained by removing the jump component $J$ from the data-generating model in Section \ref{se:simu1},  for generating the sample data $(X_{t_{j}},Y_{t_{j}})_{j=0}^{n}$ with $t_{j}=j/n$:
\begin{align*}
dY_{t}&=\exp\left\{\frac{1}{2}X_{t}\left(\begin{array}{c} -2 \\ 3 \\ 0 \end{array}\right)\right\}dw_{t}=\exp\left\{\frac{1}{2}(-2X_{1,t}+3X_{2,t})\right\}dw_{t}, \\
X_{t_{j}}&=(X_{1,t_{j}},X_{2,t_{j}},X_{3,t_{j}})=\left(\cos\left(\frac{2j\pi}{n}\right),\sin\left(\frac{2j\pi}{n}\right),\cos\left(\frac{4j\pi}{n}\right)\right), \\
Y_{0}&=0, \quad t\in[0,1].
\end{align*}

Table \ref{candsimu_appndix} summarizes the model selection frequencies.
The dpGQBIC and GQBIC tend to select Diff 2, which is the true coefficient, with high frequency for all $n$ and $\lambda$, and this
selection frequency increases with $n$.
Moreover, although the selection frequency of Diff 2 under the HGQBIC increases with $n$, the HGQBIC tends to select Diff 6, which corresponds to a smaller coefficient than the true one, when both $n$ and $\lambda$ are too small.

\begin{table}[t]
\begin{center}
\caption{Model selection frequencies for various situations in Appendix \ref{se:app2}.}
\scalebox{0.9}[0.9]{
\begin{tabular}{l l | r r r r r r r} \hline
& & \multicolumn{7}{c}{} \\[-3mm]
dpGQBIC & $n$ & Diff 1 & Diff $2^{\ast}$ & Diff 3 & Diff 4 & Diff 5 & Diff 6 & Diff 7 \\[1mm] \hline
& & \multicolumn{7}{c}{} \\[-3mm]
$\lambda=0.01$ & $100$ & 28 & 972 & 0 & 0 & 0 & 0 & 0 \\[1mm]
& $500$ & 17 & 983 & 0 & 0 & 0 & 0 & 0 \\[1mm]
& $1000$ & 4 & 996 & 0 & 0 & 0 & 0 & 0 \\[1mm] \hline
& & \multicolumn{7}{c}{} \\[-3mm]
$\lambda=0.05$ & $100$ & 23 & 977 & 0 & 0 & 0 & 0 & 0 \\[1mm]
& $500$ & 12 & 988 & 0 & 0 & 0 & 0 & 0 \\[1mm]
& $1000$ & 1 & 999 & 0 & 0 & 0 & 0 & 0 \\[1mm] \hline
& & \multicolumn{7}{c}{} \\[-3mm]
$\lambda=0.2$ & $100$ & 2 & 998 & 0 & 0 & 0 & 0 & 0 \\[1mm]
& $500$ & 2 & 998 & 0 & 0 & 0 & 0 & 0 \\[1mm]
& $1000$ & 0 & 1000 & 0 & 0 & 0 & 0 & 0 \\[1mm] \hline
& & \multicolumn{7}{c}{} \\[-3mm]
HGQBIC & $n$ & Diff 1 & Diff $2^{\ast}$ & Diff 3 & Diff 4 & Diff 5 & Diff 6 & Diff 7 \\[1mm] \hline
& & \multicolumn{7}{c}{} \\[-3mm]
$\lambda=0.01$ & $100$ & 0 & 0 & 0 & 0 & 0 & 1000 & 0 \\[1mm]
& $500$ & 0 & 348 & 0 & 0 & 0 & 652 & 0 \\[1mm]
& $1000$ & 0 & 1000 & 0 & 0 & 0 & 0 & 0 \\[1mm] \hline
& & \multicolumn{7}{c}{} \\[-3mm]
$\lambda=0.05$ & $100$ & 0 & 426 & 0 & 0 & 0 & 574 & 0 \\[1mm]
& $500$ & 0 & 1000 & 0 & 0 & 0 & 0 & 0 \\[1mm]
& $1000$ & 0 & 1000 & 0 & 0 & 0 & 0 & 0 \\[1mm] \hline
& & \multicolumn{7}{c}{} \\[-3mm]
$\lambda=0.2$ & $100$ & 0 & 992 & 0 & 0 & 0 & 8 & 0 \\[1mm]
& $500$ & 0 & 1000 & 0 & 0 & 0 & 0 & 0 \\[1mm]
& $1000$ & 0 & 1000 & 0 & 0 & 0 & 0 & 0 \\[1mm] \hline
& & \multicolumn{7}{c}{} \\[-3mm]
GQBIC & $n$ & Diff 1 & Diff $2^{\ast}$ & Diff 3 & Diff 4 & Diff 5 & Diff 6 & Diff 7 \\[1mm] \hline
& & \multicolumn{7}{c}{} \\[-3mm]
& $100$ & 34 & 966 & 0 & 0 & 0 & 0 & 0 \\[1mm]
& $500$ & 19 & 981 & 0 & 0 & 0 & 0 & 0 \\[1mm]
& $1000$ & 5 & 995 & 0 & 0 & 0 & 0 & 0 \\[1mm] \hline
\end{tabular}
}
\label{candsimu_appndix}
\end{center}
\end{table}

\end{document}